\documentclass[a4paper,11pt,reqno]{article}

\usepackage[utf8]{inputenc}
\usepackage[T1]{fontenc}
\usepackage{lmodern}
\usepackage[english]{babel}
\usepackage{microtype}

\usepackage{amsmath,amssymb,amsfonts,amsthm}
\usepackage{mathtools,accents}
\usepackage{mathrsfs}
\usepackage{aliascnt}
\usepackage{braket}
\usepackage{bm}

\usepackage[a4paper,margin=3cm]{geometry}
\usepackage[citecolor=blue,colorlinks]{hyperref}

\usepackage{enumerate}
\usepackage{xcolor}

\makeatletter
\g@addto@macro\@floatboxreset\centering
\makeatother

\makeatletter
\def\newaliasedtheorem#1[#2]#3{
  \newaliascnt{#1@alt}{#2}
  \newtheorem{#1}[#1@alt]{#3}
  \expandafter\newcommand\csname #1@altname\endcsname{#3}
}
\makeatother

\numberwithin{equation}{section}

\newtheoremstyle{slanted}{\topsep}{\topsep}{\slshape}{}{\bfseries}{.}{.5em}{}

\theoremstyle{plain}
\newtheorem{theorem}{Theorem}[section]
\newaliasedtheorem{proposition}[theorem]{Proposition}
\newaliasedtheorem{lemma}[theorem]{Lemma}
\newaliasedtheorem{corollary}[theorem]{Corollary}
\newaliasedtheorem{counterexample}[theorem]{Counterexample}

\theoremstyle{definition}
\newaliasedtheorem{definition}[theorem]{Definition}
\newaliasedtheorem{question}[theorem]{Question}
\newaliasedtheorem{example}[theorem]{Example}

\theoremstyle{remark}
\newaliasedtheorem{remark}[theorem]{Remark}

\newcommand{\setQ}{\mathbb{Q}}
\newcommand{\setR}{\mathbb{R}}

\let\altphi\phi
\let\phi\varphi
\let\varphi\altphi
\let\altphi\undefined

\let\div\undefined
\DeclareMathOperator{\div}{div}

\newcommand{\di}{\mathop{}\!\mathrm{d}}

\newcommand{\bs}{{\rm bs}}
\newcommand{\loc}{{\rm loc}}

\newcommand{\res}{\mathop{\hbox{\vrule height 7pt width .5pt depth 0pt
\vrule height .5pt width 6pt depth 0pt}}\nolimits}

\newcommand{\restr}{\raisebox{-.1618ex}{$\bigr\rvert$}}
\DeclareMathOperator{\supp}{supp}
\newcommand{\rmD}{{\rm D}}
\newcommand{\Ch}{{\sf Ch}}

\DeclareMathOperator{\Lip}{Lip}
\DeclareMathOperator{\Lipb}{Lip_b}

\DeclareMathOperator{\Cb}{C_b}
\DeclareMathOperator{\Cp}{C_p}
\DeclareMathOperator{\Cbs}{C_\bs}

\newcommand{\haus}{\mathscr{H}}
\newcommand{\leb}{\mathscr{L}}
\newcommand{\Meas}{\mathscr{M}}
\newcommand{\Prob}{\mathscr{P}}
\newcommand{\Borel}{\mathscr{B}}

\newcommand{\bb}{{\boldsymbol{b}}}

\newcommand{\dist}{\mathsf{d}}

\newcommand{\meas}{\mathfrak{m}}

\newcommand{\Algebra}{{\mathscr A}}

\DeclareMathOperator{\RCD}{RCD}
\newfont{\tmpf}{cmsy10 scaled 2500}
\newcommand\perone{\mathop{\mbox{\tmpf $\Pi$}}}

\begin{document}
\title{New stability results for sequences of metric measure spaces\\ with uniform Ricci bounds from below}
\author{Luigi Ambrosio
\thanks{Scuola Normale Superiore, \url{luigi.ambrosio@sns.it}} \and
Shouhei Honda
\thanks{Tohoku University, \url{shonda@m.tohoku.ac.jp}}} 
\maketitle

\begin{abstract}
The aim of this paper is to provide new stability results for sequences of metric measure spaces
$(X_i,\dist_i,\meas_i)$ convergent in the measured Gromov-Hausdorff sense. By adopting the so-called
extrinsic approach of embedding all metric spaces into a common one $(X,\dist)$, 
we extend the results of \cite{GigliMondinoSavare13} by providing Mosco convergence of Cheeger's energies
and compactness theorems in the whole range of Sobolev spaces $H^{1,p}$, including the space $BV$, and even
with a variable exponent $p_i\in [1,\infty]$. In addition,
building on \cite{AmbrosioStraTrevisan}, we provide local convergence results for gradient derivations. We use these tools to 
improve the spectral stability results, previously known for $p>1$ and for Ricci limit spaces, getting continuity of Cheeger's constant.
In the dimensional case $N<\infty$, we improve some rigidity and almost rigidity
results in \cite{Ketterer15a,Ketterer15b,CavallettiMondino15a,CavallettiMondino15b}.  On the basis of the second-order calculus
in \cite{Gigli}, in the class of $RCD(K,\infty)$ spaces we provide stability results for Hessians and $W^{2,2}$ functions and
we treat the stability of the Bakry-\'Emery condition $BE(K,N)$ and of ${\bf Ric}\geq KI$, with $K$ and $N$ not necessarily
constant.

\end{abstract}

\tableofcontents

\section{Introduction}

In this paper we establish new stability properties for sequences of metric measure spaces $(X,\dist_i,\meas_i)$
convergent in the measured Gromov-Hausdorff sense (mGH for short). Even though some
results are valid under weaker assumptions, to give a unified treatment of the several topics treated in this paper
we confine our discussion to sequences of $RCD(K,\infty)$ metric measure spaces, with $K\in\setR$ independent of $i$.
A pointed mGH limit of a sequence of Riemannian manifolds with a uniform lower Ricci curvature bound, called 
Ricci limit space, gives a typical example of $RCD(K,\infty)$ metric measure space, and 
this paper provides new results even for such sequences and for the corresponding Ricci limit spaces.  
Our stability results, relative to spectral properties and Hessians, extend the ones in \cite{Honda1}, \cite{Honda3} 
for compact Ricci limit spaces.  

The stability of the curvature-dimension conditions has been treated in the seminal papers \cite{LottVillani}, \cite{Sturm06},
while stability of the ``Riemannian'' condition (i.e. the quadratic character of Cheeger's energy) has been estabilished
in \cite{AmbrosioGigliSavare14}. It is by now quite clear that the treatment of stability of more complex objects derived from the metric measure
structure, like derivations, Lagrangian flows associated to derivations,
heat flows, Hessians, etc. is possible (even though we do not exclude other possibilities) by adopting the so-called
extrinsic approach, i.e. assuming that $(X_i,\dist_i)=(X,\dist)$ are independent of $i$, and that $\meas_i$ weakly converge to $\meas$
in duality with $\Cbs(X)$, the space of continuous functions with bounded support. 
We follow this approach, also because this paper builds upon the recent papers \cite{GigliMondinoSavare13}
(for stability of heat flows and Mosco convergence of Cheeger's energies) and \cite{AmbrosioStraTrevisan} 
(for strong convergence of derivations) which use the same one. See also \cite[Theorem~3.15]{GigliMondinoSavare13}
for a detailed comparison between the extrinsic approach and other intrinsic ones, with or without doubling assumptions.
In a broader context, see also the recent monograph \cite{Shioya} for the detailed analysis of convergence and concentration for metric
measure structures.

Before passing to a more precise technical description of the content of the paper, we discuss the main applications:

\smallskip\noindent
{\bf Spectral gap.} We discuss the joint continuity with respect to
$\bigl(p,(X,\dist,\meas)\bigr)$ of the $p$-spectral gaps
\begin{equation}\label{eq;p-spec gap}
\bigl(\lambda_{1,p}(X,\dist,\meas)\bigr)^{1/p}
\end{equation}
w.r.t. the mGH convergence. Here, for $p\in [1,\infty)$, $\lambda_{1, p}$ is the first positive eigenvalue of the $p$-Laplacian when $p>1$, and
Cheeger's constant when $p=1$, see \eqref{eq:def p-lap} for the precise definition in our setting. 
This extends the analysis of \cite{GigliMondinoSavare13} from $p=2$
to general $p$ and even to the case when $p$ is variable, see Theorem~\ref{thm:p-spec} and also Theorem~\ref{thm:gros}, dealing with the
case $p_i\to\infty$, with
\begin{equation}\label{eq;p-spec gap1}
\bigl(\lambda_{1, \infty}(X,\dist,\meas)\bigr)^{1/\infty}:=\frac{2}{\mathrm{diam}\,(\mathrm{supp}\,\meas)}.
\end{equation}
These general continuity properties were conjectured in \cite{Honda1} in the Ricci limit setting, and so we provide
an affirmative answer to the conjecture in the more general setting of $RCD(K,\infty)$ spaces. In particular,
Theorem~\ref{thm:p-spec} yields that Cheeger's  constants are continuous w.r.t. the mGH convergence.

The class $RCD^*(K,N)$ of metric measure spaces has been proposed in \cite{Gigli1} and 
deeply investigated in \cite{AmbrosioGigliSavare15}, \cite{ErbarKuwadaSturm} and \cite{AmbrosioMondinoSavare} in the nonsmooth
setting. Recall that in the class of smooth weighted $n$-dimensional Riemannian manifolds 
$(M^n,\dist,e^{-V}{\rm vol}_{M^n})$ the $RCD^*(K,N)$ condition, $n\leq N$, 
is equivalent to
$$
{\rm Ric}+\mathrm{Hess}(V)-\frac{\nabla V\otimes\nabla V}{N-n}\geq K I.
$$
Analogously, it is well-known that the condition $RCD(K,\infty)$ for
$(M^n,\dist,e^{-V}{\rm vol}_{M^n})$ 
is equivalent to ${\rm Ric}+\mathrm{Hess}(V)\geq KI$.

By combining the continuity of \eqref{eq;p-spec gap} with the compactness property of the class of $RCD^*(K, N)$-spaces 
w.r.t. the mGH convergence, we also establish a uniform bound
\begin{equation}\label{eq;estimate}
C_1 \le \bigl(\lambda_{1, p}(X,\dist,\meas)\bigr)^{1/p}\le C_2,
\end{equation}
where $C_i$ are positive constants depending only on $K$, $N<\infty$, and two-sided bounds of the diameter, i.e. 
$C_i$ do \textit{not} depend on $p$ (Proposition~\ref{cor:unif est p-lap}).

\smallskip\noindent
{\bf Suspension theorems.} The second application is related to almost spherical suspension theorems of positive Ricci curvature.
For simplicity we discuss here only the case when $N\geq 2$ is an integer, but our results (as those in \cite{Sturm06}, \cite{Ketterer15a},
\cite{Ketterer15b}, \cite{CavallettiMondino15b}) cover also the case $N\in (1,\infty)$.
In \cite{CavallettiMondino15b} Cavalletti-Mondino proved that for any $RCD^*(N-1, N)$-space, the quantity \eqref{eq;p-spec gap} is greater than or 
equal than $\bigl(\lambda_{1,p}(\mathbf{S}^N,\dist,\meas_N)\bigr)^{1/p}$ for any $p \in [1, \infty)$, 
where $\mathbf{S}^N$ is the unit sphere in $\setR^{N+1}$, 
$\dist$ is the standard metric of the sectional curvature $1$, and $\meas_N$ is the $N$-dimensional Hausdorff measure.
Moreover, equality implies that the metric measure space is isomorphic to a spherical suspension.
Under our notation \eqref{eq;p-spec gap1} as above, this observation is also true 
when $p=\infty$, which corresponds to the Bonnet-Myers theorem in our setting (see \cite{Sturm06} by Sturm).
Note that \cite{CavallettiMondino15b} also provides rigidity results as the following one: for a fixed $p \in [1, \infty]$, 
if $(\lambda_{1, p})^{1/p}$ is close to  
$\bigl(\lambda_{1,p}(\mathbf{S}^N,\dist,\meas_N)\bigr)^{1/p}$, then the space is Gromov-Hausdorff close to 
the spherical suspension of a compact metric space, a so-called \textit{almost} spherical suspension theorem.
The converse is known for $p \in \{2, \infty \}$ in \cite{Ketterer15a, Ketterer15b} by Ketterer and we extend the result to general $p$; in addition,
combining this with the joint spectral continuity result we can remove the
$p$-dependence in the almost spherical suspension theorem, i.e. if $(\lambda_{1, p})^{1/p}$ is close to  
$\bigl(\lambda_{1,p}(\mathbf{S}^N,\dist,\meas_N)\bigr)^{1/p}$ for \textit{some} $p\in [1,\infty]$, then this happens for any other $q\in [1,\infty]$, 
see Corollary~\ref{cor;shperical suspension}. 
This seems to be new even for compact $n$-dimensional Riemannian manifolds endowed with the $n$-dimensional Hausdorff measure.
In particular by using Petrunin's compatibility result \cite{Petrunin} between Alexandrov spaces and curvature-dimension conditions, 
this also holds for all finite-dimensional Alexandrov spaces with curvature bounded below by $1$, which is also new.

\smallskip\noindent
{\bf Stability of Hessians and of Gigli's measure-valued Ricci tensor.} The final application deals with
stability of Hessians and Ricci tensor with respect to mGH-convergence.
These notions come from the second order differential calculus on $RCD(K,\infty)$ spaces fully
developed by Gigli in \cite{Gigli}, starting from ideas from $\Gamma$-calculus. 
For Ricci limit spaces, analogous stability results were estabilished in \cite{Honda3}.
In this respect, the main novelty of this paper is the treatment of $RCD(K, \infty)$ spaces,
dropping also the dimensionality assumption. The main results are the stability of Hessians, see Corollary~\ref{thm;compactness lap} and
Corollary~\ref{thm;stability hessian}, and a kind of localized stability of the measure-valued Ricci tensor. In connection
with the latter, specifically, we prove in Theorem~\ref{thm; upp ricci} that local
lower bounds of the form
$$
\textbf{Ric}(\nabla f,\nabla f)\geq \zeta|\nabla f|^2\meas,
$$
with $\zeta\in C(X)$ bounded from below, are stable under mGH-convergence. This way, also nonconstant bounds 
from below on the Ricci tensor can be proved to be stable (see also \cite{Ketterer15c} for stability results in the same spirit,
obtained from a localization of the Lagrangian definition of curvature/dimension bounds). 
On the other hand, since our approach is extrinsic, this result becomes
of interest from the intrinsic point of view only when $\zeta$'s depending on the metric structure, as $\phi\circ\dist$,
are considered. See also Remark~\ref{rem:beKNstable} for an analogous stability property of the $BE(K,N)$
condition with $K$ and $N$ dependent on $x$.

We believe that these stability results and the tools developed in this paper 
could be the basis for the analysis of the stability of the other calculus tools and concepts
developed in \cite{Gigli}, as exterior and covariant derivatives, Hodge laplacian, etc. However, we will not pursue this point
of view in this paper.

\smallskip\noindent
{\bf Organization of the paper.} In Section~\ref{sec:2} we introduce the main measure-theoretic preliminaries.
In Section~\ref{sec:3} we discuss convergence of functions $f_i$ in different measure spaces relative to $\meas_i$; here the main
new ingredient is a notion of $L^{p_i}$ convergence which allows us also to cover the case when the exponents
$p_i$ converge to $p\in [1,\infty)$. We discuss the case of strong convergence, and of weak convergence when $p>1$. 
Section~\ref{sec:4} recalls the main terminology and the main known facts about $RCD(K,\infty)$ spaces and the
regularizing properties of the heat flow $h_t$. Less standard facts proved in this section are: the formula provided 
in Proposition~\ref{prop:newrepGamma} for $u\mapsto\int_X|\nabla u|\di\meas$ (somehow reminiscent of the duality 
tangent/cotangent bundle at the basis of \cite{Gigli}), of particular interest for
the proof of lower semicontinuity properties, and the weak isoperimetric property of Proposition~\ref{prop:quasiiso}.

In Section~\ref{sec:5} we enter the core of the paper, somehow ``localizing'' the Mosco convergence result of Cheeger's
energies of \cite{GigliMondinoSavare13}. The main result is Theorem~\ref{thm:cont_reco} where we prove, among other things, that
the measures $|\nabla f_i|_i^2\meas_i$ weakly converge to $|\nabla f|^2\meas$ whenever $f_i$ strongly converge to
$f$ in $H^{1,2}$ (i.e., $f_i$ $L^2$-strongly converge to $f$ and the Cheeger energies of $f_i$ converge to the Cheeger energy
of $f$). To prove this, the main difficulty is the localization of the $\liminf$ inequality of \cite{GigliMondinoSavare13}; we obtain
it using the recent results in \cite{AmbrosioStraTrevisan}, for families of derivations with convergent $L^2$ norms (in this
case, gradient derivations, see Theorem~\ref{thm:strong-convergence-gradients} in this paper). Section~\ref{sec:6}
covers the stability properties of $BV$ functions, the main result is that $f\in BV(X,\dist,\meas_i)$ whenever
$f_i\in BV(X,\dist,\meas_i)$ $L^1$-strongly converge to $f$, with $L=\liminf_i|\rmD f_i|(X)<\infty$. In addition,
$|\rmD f|(X)\leq L$. The proof of this stability properties strongly relies on the results of Section~\ref{sec:5} and,
nowithstanding the well-estabilished Eulerian-Lagrangian duality for Sobolev and $BV$ spaces (see \cite{AmbrosioDiMarino}
for the latter spaces) it seems harder to get from the Lagrangian point of view.

Section~\ref{sec:7} covers compactness results
for $BV$ and $H^{1,p}$, also in the case when $p$ depends on $i$. In the proof of these facts we use the (local)
strong $L^2$ compactness properties for sequences bounded $H^{1,2}$ proved in \cite{GigliMondinoSavare13}; passing from 
the exponent 2 to higher exponents is quite simple, while the treatment of smaller powers and the passage from $L^p_\loc$
to $L^p$ convergence (essential for our results in Section~\ref{sec:9}) requires the existence of uniform isoperimetric
profiles. We review the state of the art on this topic in Theorem~\ref{thm:isopro}. 
In Section~\ref{sec:8} we prove $\Gamma$-convergence of the $p_i$-Cheeger energies 
$\Ch^i_{p_i}$ relative to $(X,\dist,\meas_i)$ (set equal to the total variation functional
$f\mapsto |\rmD f|(X)$ in $BV$ when $p=1$), namely 
$$
\liminf_{i\to\infty}\Ch^i_{p_i}(f_i)\geq\Ch_p(f)
$$
whenever $f_i$ $L^{p_i}$-strongly converge to $f$, and the existence of a sequence $f_i$ with this property
satisfying $\limsup_i\Ch^i_{p_i}(f_i)\leq \Ch_p(f)$.
The only difference with the case $p=2$ considered in \cite{GigliMondinoSavare13} is that, in general, we are not able to achieve
the $\liminf$ inequality  with $L^{p_i}$-weakly convergent sequences, unless a uniform isoperimetric assumption on the spaces
grants relative compactness w.r.t. strong $L^{p_i}$ convergence. Under this assumption, Mosco and $\Gamma$-convergence coincide.

Finally, Section~\ref{sec:9}, Section~\ref{sec:10} and Section~\ref{sec:11} cover the above mentioned stability results for
$p$-eigenvalues and eigenfunctions (using Section~\ref{sec:7} and Section~\ref{sec:8}) , for Hessians and Ricci tensors
(using Section~\ref{sec:5}), and the dimensional results relative to the suspension theorems (using Section~\ref{sec:9}).

\smallskip
{\bf Acknowledgement.} The first author acknowledges helpful conversations on the subject of this paper with Fabio Cavalletti,
Andrea Mondino and Giuseppe Savar\'e. The second author acknowledges the support of the JSPS Program for Advancing 
Strategic International Networks to Accelerate the Circulation of Talented Researchers, the
Grant-in-Aid for Young Scientists (B) 16K17585 and the warm hospitality of SNS.\\
The authors warmly thank the referee for the detailed reading of the paper and for the constructive comments.

\section{Notation and basic setting}\label{sec:2}

\noindent{\bf Metric concepts.} In a metric space $(X,\dist)$, we denote by $B_r(x)$ and $\overline{B}_r(x)$ the open and closed
balls respectively, by $\Cbs(X)$ the space of bounded 
continuous functions with bounded support, by $\Lip_\bs(X)\subset \Cbs(X)$ the subspace of Lipschitz 
functions. We use the notation $\Cb(X)$ and $\Lipb(X)$ for bounded continuous and bounded Lipschitz functions respectively.

For $f:X\to\setR$ we denote by $\Lip(f)\in [0,\infty]$ the Lipschitz constant and by ${\rm lip}(f)$ the slope, namely
\begin{equation}\label{eq:localsl}
{\rm lip}(f)(x):=\limsup_{y\to x}\frac{|f(y)-f(x)|}{\dist(y,x)}.
\end{equation}
We also define the asymptotic Lipschitz constant by
\begin{equation}\label{eq:deflipa}
{\rm Lip}_a f(x) = \inf_{r>0} \Lip \bigl( f\restr_{B_r(x)} \bigr) =
\lim_{r\to0^+} \Lip \bigl( f\restr_{B_r(x)} \bigr),
\end{equation}
which is upper semicontinuous.

\smallskip
\noindent{\bf The metric algebra $\Algebra_\bs$.} We associate to any separable metric space $(X,\dist)$ the smallest 
$\Algebra\subset\Lipb(X)$ containing
\begin{equation}\label{eq:setD}
\min\{\dist(\cdot,x),k\}\quad\text{with $k\in\setQ\cap [0,\infty]$, $x\in D$ and $D\subset X$ countable and dense}
\end{equation}
which is a vector space over $\setQ$ and is stable under products and lattice operations.
It is a countable set and it depends only on the choice of the set $D$ (but this dependence will not be emphasized 
in our notation, since the metric space will mostly be fixed). We shall work with the subalgebra $\Algebra_\bs$ of functions with bounded support.

\smallskip
\noindent{\bf Measure-theoretic notation.} The Borel $\sigma$-algebra of a metric space $(X,\dist)$ is denoted $\Borel(X)$.
The Borel signed measures with finite total variation are denoted by $\Meas(X)$, while we use the notation $\Meas^+(X)$, $\Meas^+_\loc(X)$, $\Prob(X)$ for 
nonnegative finite Borel measures, Borel measures which are finite on bounded sets and Borel probability measures. 

We use the standard notation $L^p(X,\meas)$, $L^p_\loc(X,\meas)$ for the $L^p$ spaces
when $\meas$ is nonnegative ($p=0$ is included  and denotes the class of $\meas$-measurable functions). 
Notice that, in this context where no local compactness assumption is made, $L^p_\loc$ means
$p$-integrability on bounded subsets.

Given metric spaces $(X,\dist_X)$ and $(Y,\dist_Y)$ and a Borel map $f:X\to Y$, we denote by $f_\#$ the induced push-forward
operator, mapping $\Prob(X)$ to $\Prob(Y)$, $\Meas^+(X)$ to $\Meas^+(Y)$ and, if the preimage of bounded sets is bounded, 
$\Meas^+_\loc(X)$ to $\Meas^+_\loc(Y)$. Notice that, for all $\mu\in\Meas^+(X)$, $f_\#\mu$ is well defined also if $f$ is $\mu$-measurable.

\smallskip
\noindent{\bf Convergence of measures.} 
We say that $\meas_n\in\Meas_\loc(X)$ weakly converge to $\meas\in\Meas_\loc(X)$ if
$\int_X v\di\meas_n\to\int_X v\di\meas$ as $n\to\infty$ for all $v\in \Cbs(X)$. When all the measures $\meas_n$ as well as $\meas$
are probability measures, this is equivalent to requiring that $\int_X v\di\meas_n\to\int_X v\di\meas$ as $n\to\infty$ for all $v\in\Cb(X)$. 
We shall also use the following well-known proposition.

\begin{proposition}\label{prop:passa_al_limite} If $\meas_n$ weakly converge to $\meas$ in $\Meas^+_\loc(X)$, and if
$\limsup_{i\to\infty}\int_X\Theta\di\meas_i<\infty$ for some Borel $\Theta:X\to (0,\infty]$, then 
\begin{equation}\label{eq:passa}
\lim_{i\to\infty}\int_X v\di\meas_i=\int_X v\di\meas
\end{equation} 
for all $v:X\to\setR$ continuous with $\lim_{\dist(x,\bar x)\to\infty}|v|(x)/\Theta(x)=0$
for some (and thus all) $\bar x\in X$. If $\Theta:X\to [0,\infty)$ is continuous and
$$\limsup_{n\to\infty}\int_X\Theta\di\meas_n\leq\int_X\Theta\di\meas<\infty,$$ 
then \eqref{eq:passa} holds for all $v:X\to\setR$ continuous with
$|v|\leq C\Theta$ for some constant $C$.
\end{proposition}

\smallskip
\noindent{\bf Metric measure space.} Throughout this paper, a \emph{metric measure space} is a triple $(X,\dist,\meas)$, where $(X,\dist)$
is a complete and separable metric space and $\meas\in\Meas_\loc^+(X)$. 

As explained in the introduction, in this paper we always consider metric measure spaces according to the
previous definition. When a sequence convergent in the measured-Gromov Hausdorff sense is considered,
we shall always assume (up to an isometric embedding in a common space) that the sequence has the structure 
$(X,\dist,\meas_i)$ with $\meas_i\in \Meas_\loc^+(X)$ weakly convergent to $\meas\in\Meas_\loc^+(X)$. In particular,
this convention forces us to drop the condition $\supp\meas =X$, used in many papers where individual spaces
are considered.

\section{Convergence of functions}\label{sec:3}

In our setting, we are dealing with a sequence $(\meas_i)\subset\Meas_\loc^+(X)$ weakly convergent to $\meas\in\Meas^+_\loc(X)$.
Assuming that $f_i$ in suitable Lebesgue spaces relative to $\meas_i$ are given, we discuss in this section suitable notions
of weak and strong convergence for $f_i$. 
Motivated by the convergence results of Section~\ref{sec:8} and Section~\ref{sec:9}, we extend the analysis of 
\cite{GigliMondinoSavare13} and \cite{AmbrosioStraTrevisan} to the case when also the exponents $p_i\in [1,\infty)$
are allowed to vary, with $p_i\to p\in [1,\infty)$. For weak convergence we only consider the case $p>1$ (we don't
need $L^1$-weak convergence), while for strong convergence, in connection with the results of Section~\ref{sec:6},
we also consider the case $p=1$.

\smallskip\noindent
{\bf Weak convergence.} Assume that $p_i\in [1,\infty)$ converge to $p\in (1,\infty)$. We say that $f_i\in L^{p_i}(X,\meas_i)$ 
$L^{p_i}$-weakly converge to $f \in L^p(X,\meas)$ if $f_i\meas_i$ weakly converge to $f \meas$ in $\Meas_\loc(X)$, with
\begin{equation}\label{eq:bound-lp-norms} 
\limsup_{i\to \infty} \left\| f_i \right\|_{L^{p_i}(X,\meas_i)} < \infty.
\end{equation}
For $\setR^k$-valued maps we understand the convergence componentwise.

It is obvious that $L^{p_i}$-weak convergence is stable under finite sums.
The proof of the following result is very similar to the proof in the case when $p$ and $\meas$ are fixed, and it omitted.

\begin{proposition}\label{prop:weakcon}
If $f_i\in L^{p_i}(X,\meas_i;\setR^k)$  $L^{p_i}$-weakly converge to $f \in L^p(X,\meas;\setR^k)$, then 
$$\| f\|_{L^p(X,\meas;\setR^k)} \le  \liminf_{i\to\infty}\| f_i\|_{L^{p_i}(X,\meas_i;\setR^k)}.$$ 
Moreover, any sequence $f_i\in L^{p_i}(X,\meas_i;\setR^k)$ such that \eqref{eq:bound-lp-norms} holds
admits a $L^{p_i}$-weakly convergent subsequence.
\end{proposition}
 
\smallskip\noindent
{\bf Strong convergence.} We discuss the simpler case $p_i=p$ first. If $p>1$ we say that $f_i\in L^p(X,\meas_i;\setR^k)$  
$L^p$-strongly converge to $f\in L^p(X,\meas;\setR^k)$ if, in addition to weak $L^p$-convergence, one has 
$\limsup_i\|f_i\|_{L^p(X,\meas_i;\setR^k)}\le\|f\|_{L^p(X,\meas;\setR^k)}$. If $k=p=1$, we say that $f_i\in L^1(X,\meas_i)$  
$L^1$-strongly converge to $f\in L^1(X,\meas)$ if $\sigma\circ f_i$ $L^2$-strongly converges to $\sigma\circ f$,
where $\sigma(z)={\rm sign}(z)\sqrt{|z|}$ is the signed square root.

In the following remark we see that strong convergence can be written in terms of convergence of the probability
measures naturally associated to the graphs of $f_i$; this holds also for vector valued maps 
and we will use this fact in the proof of Proposition~\ref{prop:strocon}.

\begin{remark}[Convergence of graphs versus $L^p$-strong convergence]\label{rem:sono_simili}
{\rm If $p>1$ one can use the strict convexity of the map $z\in\setR^k\mapsto |z|^p$ to prove that 
$F_i:X\to\setR^k$ $L^p$-strongly converge to $F$ if and only if   
$\mu_i = (Id\times F_i)_\#\meas_i$ weakly converge to $\mu = (Id\times F)_\#\meas$
in duality with 
\begin{equation}\label{eq:pspace}
\Cp(X\times\setR^k):=\left\{\psi\in C(X\times\setR^k):\ |\psi(x,z)|\leq C|z|^p\,\,\text{for some $C\geq 0$}\right\}
\end{equation}
(see for instance \cite[Section~5.4]{AmbrosioGigliSavare08}, \cite{GigliMondinoSavare13}). \\
If $p=k=1$, we can use the fact that the signed square root is an homeomorphism of $\setR$ and the equivalence
estabilished in the quadratic case to get the same result. 
}\end{remark}

We recall in the following proposition a few well-known properties of $L^p$-strong convergence, see also
\cite{Honda2}, \cite{GigliMondinoSavare13} for a more detailed treatment of this topic.

\begin{proposition}\label{prop:strocon}
For all $p\in [1,\infty)$ the following properties hold:
\begin{itemize}
\item[(a)]  If $f_i$ $L^p$-strongly converge to $f$ the functions $\phi\circ f_i$ $L^p$-strongly converge to $\phi\circ f$
for all $\phi\in\Lip(\setR)$ with $\phi(0)=0$. 
\item[(b)] If $f_i,\,g_i$ $L^p$-strongly converge to $f,\,g$ respectively, then $f_i+g_i$ $L^p$-strongly converge to $f+g$. 
\item[(c)] If $f_i$ $L^p$-strongly (resp. $L^p$-weakly) converge to $f$, then $\varphi f$ $L^p$-strongly (resp. $L^p$-weakly)
converge to $\varphi f$ for all $\varphi\in\Cb(X)$ (resp. $\varphi\in\Cbs(X)$).
\item[(d)] If $f_i$ $L^2$-strongly converge to $f$ and $g_i$ $L^2$-weakly converge to $g$, then 
$$
\lim_{i\to\infty}\int_X f_ig_i\di\meas_i=\int_X fg\di\meas.
$$
If $g_i$ are also $L^2$-strongly convergent, then $f_ig_i$ are $L^1$-strongly convergent.
\item[(e)] If $(g_i)$ is uniformly bounded in $L^\infty$ and $L^1$-strongly convergent to $g$, then
$$
\lim_{i\to\infty}\|g_i\|_{L^{p_i}(X,\meas_i)}=\|g\|_{L^p(X,\meas)}
$$
whenever $p_i\in [1,\infty)$ converge to $p\in [1,\infty)$.
\end{itemize}
\end{proposition}  
\begin{proof} (a) In the case $p>1$ this is a simple consequence of Remark~\ref{rem:sono_simili}, since
$\mu_i = (Id\times f_i)_\#\meas_i$ weakly converge to $\mu = (Id\times f)_\#\meas$ in duality with
the space in duality with the space $\Cp(X\times\setR)$ in \eqref{eq:pspace}. Since 
$\tilde\psi(x,z)=\psi(x,\phi(z))$ belongs to $\Cp(X\times\setR)$ for all $\psi\in \Cp(X\times\setR)$ it follows that
$(Id\times \phi\circ f_i)_\#\meas_i$ weakly converge to $\mu = (Id\times \phi\circ f)_\#\meas$ in duality with
$\Cp(X\times\setR)$, and then Remark~\ref{rem:sono_simili} applies again to provide the $L^p$-strong
convergence of $\phi\circ f_i$ to $\phi\circ f$.

In the case $p=1$,
since $\sigma(\phi(f_i))={\rm sign}(\phi\circ f_i)\sqrt{|\phi|\circ f_i}$, from the strong $L^2$-convergence of $\sqrt{\phi^\pm\circ f_i}$
to $\sqrt{\phi^\pm\circ f}$ and the additivity of $L^2$-strong convergence (proved in (b)) we get the result.

(b) The case $p>1$ is dealt with, for instance, in \cite{Honda2}, see Corollary~3.26 and Proposition~3.31 therein. 
In order to prove additivity for $p=1$ we can reduce ourselves, thanks to the stability under left composition proved
in (a), to the sum of nonnegative functions $u_i,\,v_i$. Since $\sqrt{u_i}$ and $\sqrt{v_i}$ are $L^2$-strongly
convergent, using the identity $\sqrt{u_i+v_i}=\sqrt{\sqrt{u_i}^2+\sqrt{v_i}^2}$ we obtain that also $\sqrt{u_i+v_i}$
is strongly $L^2$-convergent.

The proof of (c) is a simple consequence of the definitions of $L^p$-strong convergence, splitting $\phi$ and $f_i$ in positive
and negative parts to deal also with the case $p=1$.

The proof of the first part of statement (d) is a simple consequence of 
$$\liminf_i\|f_i+tg_i\|_{L^2(X,\meas_i)}\geq\|f+tg\|_{L^2(X,\meas)}\qquad \forall t\in\setR,$$ 
see also Section~\ref{sec:8} where a similar argument is used in connection with Mosco convergence. In order to
prove $L^1$-strong convergence when also $g_i$ are $L^2$-strongly convergent, we can reduce ourselves to the case when $f_i$ and
$g_i$ are nonnegative. Then, convergence of the $L^2$ norms of $\sqrt{f_ig_i}$ follows by the first part of the statement; weak convergence
of $\sqrt{f_ig_i}\meas_i$ to $\sqrt{fg}\meas$ follows by Remark~\ref{rem:sono_simili}, with $k=p=2$, $F_i=(f_i,g_i)$ and 
$\psi(z)=\sqrt{|z_1||z_2|}$.

For the proof of (e), let $N=\sup_i\|g_i\|_{L^\infty(X,\meas_i)}$ and notice first that $(g_i)$ is uniformly bounded in $L^{p_i}$. Hence,
the $\liminf$ inequality follows by the $L^{p_i}$-weak convergence of $g_i$ to $g$. The proof of the 
$\limsup$ inequality follows by statement (a) with $\phi(z)=|z|^p\land N^p$, which ensures that
$\int_X\phi(g_i)\di\meas_i\to\int_X\phi(g)\di\meas=\|g\|^p_{L^p(X,\meas)}$,
noticing that $p_i\to p$ implies $\int_X\phi(g_i)\di\meas_i-\int_X |g_i|^{p_i}\di\meas_i\to 0$.
\end{proof}

Now we turn to the general case $p_i\to p\in [1,\infty)$. We say that $L^{p_i}$-strongly converge to $f$ 
if $f_i\in L^{p_i}(X,\meas_i)$, $L^{p_i}$-weakly convergent to $f\in L^p(X,\meas)$ and if 
for any $\epsilon>0$ we can find an additive decomposition $f_i=g_i+h_i$ with
\begin{itemize}
\item[(i)] $(g_i)$ uniformly bounded in $L^\infty$, and strongly $L^1$-convergent;
\item[(ii)] $\sup_i\|h_i\|_{L^{p_i}(X,\meas_i)}<\epsilon$.
\end{itemize}

It is obvious from the definition that also $L^{p_i}$-strong convergence is stable under finite sums. In the following
proposition we show that stability under composition with Lipschitz maps $\phi$ holds and that $L^{p_i}$ convergence implies
convergence of the $L^{p_i}$ norms.

\begin{proposition} [Properties of $L^{p_i}$-strong convergence] The following properties hold:\break
\begin{itemize}
\item[(a)] If $f_i$ $L^{p_i}$-strongly converge to $f$, the functions $\phi\circ f_i$ $L^{p_i}$-strongly converge to $\phi\circ f$
for all $\phi\in\Lip(\setR)$ with $\phi(0)=0$. 
\item[(b)] If $(f_i)$ is $L^{p_i}$-strongly convergent to
$f\in L^p(X,\meas)$, then
$$ \lim_{i\to\infty}\| f_i\|_{L^{p_i}(X,\meas_i)}=\| f\|_{L^p(X,\meas)}.$$
\end{itemize}
\end{proposition}
\begin{proof} (a) Possibly splitting $\phi$ in positive and negative part we can assume $\phi\geq 0$. Since $\phi$ is a contraction,
taking also Proposition~\ref{prop:strocon}(a) into account, it is immediate to check that decompositions $f_i=g_i+h_i$ induce
decompositions $\phi\circ g_i+(\phi\circ f_i-\phi\circ g_i)$ of $\phi\circ f_i$; in addition, if $\psi$ is any $L^{p_i}$-weak limit
point of $(\phi\circ f_i)$, from the lower semicontinuity of $L^{p_i}$ convergence we get 
\begin{eqnarray*}
&&\|\psi-\phi\circ g\|_{L^p(X,\meas)}\leq\liminf\limits_{i\to\infty}\|\phi\circ h_i\|_{L^{p_i}(X,\meas_i)}\leq \Lip(\phi)\epsilon\\
&&\|\phi\circ f-\phi\circ g\|_{L^p(X,\meas)}\leq\Lip(\phi)\|f-g\|_{L^p(X,\meas)}\leq
\Lip(\phi)\liminf\limits_{i\to\infty}\|h_i\|_{L^{p_i}(X,\meas_i)} \leq\Lip(\phi)\epsilon,
\end{eqnarray*}
where $g$ denotes the $L^{p_i}$-strong limit of $g_i$.  
Since $\epsilon$ is arbitrary, we obtain that $\psi=\phi\circ f$, and this proves the $L^{p_i}$-strong convergence of 
$f_i$ to $f$.

(b) The $\liminf$ inequality follows by weak convergence. If $f_i=g_i+h_i$ is a decomposition as in (i), (ii), and if $g$ is the $L^{p_i}$-strong limit of 
$g_i$, the $\limsup$ inequality is a direct consequence of the inequality $\|f-g\|_{L^p(X,\meas)}<\epsilon$ and of
$$
\lim_{i\to\infty}\|g_i\|_{L^{p_i}(X,\meas_i)}=\|g\|_{L^p(X,\meas)},
$$
ensured by Proposition~\ref{prop:strocon}(e).
\end{proof}

\section{Minimal relaxed slopes, Cheeger energy and $RCD(K,\infty)$ spaces}\label{sec:4}

In this section we recall basic facts about minimal relaxed slopes, Sobolev spaces and heat flow in metric measure spaces
$(X,\dist,\meas)$, see \cite{AmbrosioGigliSavare13} and \cite{Gigli1} for a more systematic treatment of this topic. For
$p\in (1,\infty)$ the $p$-th Cheeger energy
$\Ch_p:L^p(X,\meas)\to [0,\infty]$ is the convex and $L^p(X,\meas)$-lower semicontinuous functional defined as follows:
\begin{equation}\label{eq:defchp}
\Ch_p(f):=\inf\left\{\liminf_{n\to\infty}\frac 1p\int_X{\rm Lip}_a^p(f_n)\di\meas:\ \text{$f_n\in\Lipb(X)\cap L^p(X,\meas)$, $\|f_n-f\|_p\to 0$}\right\}. 
\end{equation}
The original definition in \cite{Cheeger} involves generalized upper gradients of $f_n$ in place of their asymptotic Lipschitz constant, but many
other pseudo gradients (upper gradients, or the slope ${\rm lip}(f)\leq{\rm Lip}_a(f)$, which is a particular upper gradient) can be used and all of them lead to the same definition. Indeed, all these pseudo gradients
produce functionals intermediate between the functional in \eqref{eq:defchp} and the functional based on the minimal $p$-weak upper gradient
of \cite{Shanmugalingam}, which are shown to be coincident in \cite{AmbrosioColomboDiMarino}
(see also the discussion in \cite[Remark~5.12]{AmbrosioGigliSavare13}).

The Sobolev spaces $H^{1,p}(X,\dist,\meas)$ are simply defined as the finiteness domains of $\Ch_p$. When endowed with the norm
$$
\|f\|_{H^{1,p}}:=\left(\|f\|_{L^p(X,\meas)}^p+p\Ch_p(f)\right)^{1/p}
$$
these spaces are Banach, and reflexive if $(X,\dist)$ is doubling (see \cite{AmbrosioColomboDiMarino}).  

The case $p=2$ plays an important role in the construction of the differentiable structure, following \cite{Gigli}.
For this reason we use the disinguished notation $\Ch=\Ch_2$ and it can be proved that $H^{1,2}(X,\dist,\meas)$ is Hilbert if $\Ch$ is quadratic.
  
In connection with the definition of $\Ch$, for all $f\in H^{1,2}(X,\dist,\meas)$ one can consider the collection $RS(f)$ all functions in $L^2(X,\meas)$
larger than a weak $L^2(X,\meas)$ limit of ${\rm Lip}_a(f_n)$, with $f_n\in\Lipb(X)$ and $f_n\to f$ in $L^2(X,\meas)$. 
This collection describes a convex, closed and nonempty set, 
whose element with smallest $L^2(X,\meas)$ norm is called minimal relaxed slope and denoted by $|\nabla f|$. We use the
not completely appropriate nabla notation, instead of the notation $|\rmD f|$ of \cite{Gigli}, since we will be dealing only with quadratic $\Ch$. 
Notice also that a similar construction can be applied to $\Ch_p$, and provides a minimal $p$-relaxed gradient that can indeed depend on
$p$ (see \cite{DmSp}). However, either under the doubling\&Poincar\'e assumptions \cite{Cheeger}, or under curvature assumptions
\cite{GigliHan} this dependence disappears and in any case we will only be dealing with the $2$-minimal relaxed slope in this paper.

When $\Ch$ is quadratic we denote by $\langle\nabla f,\nabla g\rangle$ the canonical symmetric bilinear form from 
$[H^{1,2}(X,\dist,\meas)]^2$ to $L^1(X,\meas)$ defined by
\begin{equation}\label{eq:defccarre}
\langle\nabla f,\nabla g\rangle:=\lim_{\epsilon\to 0}\frac{|\nabla (f+\epsilon g)|^2-|\nabla f|^2}{2\epsilon}
\end{equation}
(where the limit is understood in the $L^1(X,\meas)$ sense). Notice also that the
expression $\langle\nabla f,\nabla g\rangle$ still makes sense $\meas$-a.e. for any $f,\,g\in\Lipb(X)$ 
(not necessarily in the $H^{1,2}$ space, when $\meas(X)=\infty$), 
since $f,\,g$ coincide on bounded sets with functions in the Sobolev class, and gradients
satisfy the locality property on open and even on Borel sets.

Because of the minimality
property, $|\nabla f|$ provides integral representation to $\Ch$, so that
$$
\int_X\langle\nabla f,\nabla g\rangle\di\meas=\lim_{\epsilon\to 0}\frac{\Ch (f+\epsilon g)-\Ch(f)}{\epsilon}
$$
and it is not hard to improve weak to strong convergence.  

\begin{theorem}\label{tapprox}
For all $f\in D(\Ch)$ one has
$$
\Ch(f)=\frac 12\int_X|\nabla f|^2\di\meas
$$
and there exist $f_n\in\Lipb(X)\cap L^2(X,\meas)$ with $f_n\to f$ in $L^2(X,\meas)$ and 
${\rm Lip}_a(f_n)\to |\nabla f|$ in $L^2(X,\meas)$.  In particular, if $H^{1,2}(X,\dist,\meas)$ is reflexive,
there exist $f_n\in\Lipb(X)\cap L^2(X,\meas)$ satisfying $f_n\to f$ in $L^2(X,\meas)$ and 
$|\nabla (f_n-f)|\to 0$ in $L^2(X,\meas)$. 
\end{theorem}

Most standard calculus rules can be proved, when dealing with minimal relaxed slopes. For the purposes of this paper the most relevant ones
are:\smallskip

\noindent{\bf Locality on Borel sets.} $|\nabla f|=|\nabla g|$ $\meas$-a.e. on $\{f=g\}$ for all $f,\,g\in H^{1,2}(X,\dist,\meas)$;

\smallskip
\noindent{\bf Pointwise minimality.} $|\nabla f|\leq g$ $\meas$-a.e. for all $g\in RS(f)$;

\smallskip
\noindent{\bf Degeneracy.} $|\nabla f|=0$ $\meas$-a.e. on $f^{-1}(N)$ for all $f\in H^{1,2}(X,\dist,\meas)$ and all $\leb^1$-negligible $N\in\Borel(\setR)$;

\smallskip
\noindent{\bf Chain rule.} $|\nabla (\varphi\circ f)|=|\varphi'(f)||\nabla f|$ for all $f\in H^{1,2}(X,\dist,\meas)$ and all
$\varphi:\setR\to\setR$ Lipschitz with $\varphi(0)=0$.

\smallskip
\noindent{\bf Leibniz rule.} If $f,\,g\in H^{1,2}(X,\dist,\meas)$ and $h\in\Lipb(X)$, then
$$
\langle\nabla f,\nabla (gh)\rangle=h\langle\nabla f,\nabla g\rangle+g\langle\nabla f,\nabla h\rangle
\qquad\text{$\meas$-a.e. in $X$.}
$$
\smallskip

Another object canonically associated to $\Ch$ and then to the metric measure structure is the heat flow $h_t$, defined as the
$L^2(X,\meas)$ gradient flow of $\Ch$, according to the Brezis-Komura theory of gradient flows of lower semicontinuous
functionals in Hilbert spaces, see for instance \cite{Brezis}. This theory provides a continuous contraction
semigroup $h_t$ in $L^2(X,\meas)$ 
which, under the growth condition
\begin{equation}\label{eq:Grygorian}
\meas\bigl(B_r(\bar x)\bigr)\leq c_1 e^{c_2r^2}\qquad\forall r>0,
\end{equation}
extends to a continuous and mass preserving semigroup (still denoted $h_t$) in all
$L^p(X,\meas)$ spaces, $1\leq p<\infty$. In addition, $h_t$ preserves upper and lower
bounds with constants, namely $f\leq C$ $\meas$-a.e. (resp. $f\geq C$ $\meas$-a.e.)
implies $h_t f\leq C$ $\meas$-a.e.  (resp. $h_tf\geq C$ $\meas$-a.e.) for all $t\geq 0$.

We shall use $h_t$ only in the case when $\Ch$ is quadratic, as a regularizing operator. In the sequel 
we adopt the notation
\begin{equation}\label{eq:defDelta}
D(\Delta):=\left\{f\in H^{1,2}(X,\dist,\meas):\ \Delta f\in L^2(X,\meas)\right\}
\end{equation}
namely $D(\Delta)$ is the class of functions $f\in H^{1,2}(X,\dist,\meas)$ satisfying $-\int_X v g\di\meas=\int_X\langle\nabla f,\nabla v\rangle\di\meas$
for all $v\in H^{1,2}(X,\dist,\meas)$, for some $g\in L^2(X,\meas)$ (and then, since $g$ is uniquely determined, $\Delta f:=g$).
When $\Ch$ is quadratis the semigroup $h_t$ is also linear (and this property is equivalent to $\Ch$ being quadratic) and it is easily
seen that
$$
\lim_{t\downarrow 0} h_tf=f\qquad\text{strongly in $H^{1,2}$ for all $f\in H^{1,2}(X,\dist,\meas)$.}
$$
We shall also extensively use the typical regularizing properties (independent of curvature assumptions)
\begin{equation}\label{eq:Brezis1}
\text{$h_tf\in W^{1,2}(X,\dist,\meas)$ for all $f\in L^2(X,\meas)$, $t>0$ and $\Ch(h_t f)\leq\frac{\|f\|_{L^2(X,\meas)}^2}{2t}$,}
\end{equation}
\begin{equation}\label{eq:Brezis2}
\text{$h_tf\in D(\Delta)$ for all $f\in L^2(X,\meas)$, $t>0$ and $\|\Delta h_t f\|_{L^2(X,\meas)}^2\leq\frac{\|f\|_{L^2(X,\meas)}^2}{t^2}$,}
\end{equation}
as well as the commutation rule $h_t\circ\Delta=\Delta\circ h_t$, $t>0$.

Finally, we describe the class of $RCD(K,\infty)$ metric measure spaces of \cite{AmbrosioGigliSavare14}, where thanks to the lower
bounds on Ricci curvature even stronger properties of $h_t$ can be proved.

\begin{definition}[$CD(K,\infty)$ and $RCD(K,\infty)$ spaces] We say that a metric measure space $(X,\dist,\meas)$
satisfying the growth bound \eqref{eq:Grygorian} (for some constants $c_1,\,c_2$ and some $\bar x\in X$)
is a $RCD(K,\infty)$ metric measure space, with $K\in\setR$, if:
\begin{itemize}
\item[(a)] setting
$$
\Prob_2(X):=\left\{\mu\in\Prob(X):\ \int_X\dist^2(\bar x,x)\di\meas(x)<\infty\right\},
$$
the Relative Entropy Functional ${\rm Ent}(\mu):\Prob_2(X)\to\setR\cup\{\infty\}$ given by
\begin{equation}\label{eq:defentropy}
{\rm Ent}(\mu):=
\begin{cases}
\int_X\rho\log\rho\di\meas&\text{if $\mu=\rho\meas\ll\meas$;}
\\ 
\infty &\text{otherwise}
\end{cases}
\end{equation}
is $K$-convex along Wasserstein geodesics in $\Prob_2(X)$, namely
$$
{\rm Ent}(\mu_t)\leq (1-t){\rm Ent}(\mu_0)+t{\rm Ent}(\mu_1)-\frac{K}{2}t(1-t)W_2^2(\mu_0,\mu_1)
$$
for all $\mu_0,\,\mu_1\in D({\rm Ent}):=\{\mu:\ {\rm Ent}(\mu)<\infty\}$, 
for some constant speed geodesic $\mu_t$ from $\mu_0$ to $\mu_1$ (so, this condition forces
$D({\rm Ent}),W_2)$ to be geodesic). This condition corresponds to the $CD(K,\infty)$ condition
of \cite{LottVillani}, \cite{Sturm06}.
\item[(b)] $\Ch$ is quadratic. This is the axiom added to the Lott-Sturm-Villani theory in \cite{AmbrosioGigliSavare14}.
\end{itemize}
\end{definition}

\begin{remark} [On the growth condition \eqref{eq:Grygorian}]\label{rem:Grygorian}
Notice that \eqref{eq:Grygorian} is needed to give a meaning to the integral in \eqref{eq:defentropy}, as it ensures the 
integrability of the negative part of $\rho\log\rho$. On the other hand, adopting a suitable
convention on the meaning to be given to ${\rm Ent}$ in these cases of indeterminacy (so that
the $CD(K,\infty)$ condition makes anyhow sense), it has been proved in \cite{Sturm06} that
\eqref{eq:Grygorian} can be deduced from the $CD(K,\infty)$ condition, and that the constants $c_i$
can be estimated in terms of $K$ and of the measure of two concentric balls centered at
$\bar x\in\supp\meas$.
\end{remark}

It is not hard to prove that the support of any $\RCD(K,\infty)$ (or even $CD(K,\infty)$ space)
is length, namely the infimum of the length of the absolutely continuous curves connecting any two points $x,\,y\in\supp\meas$
is $\dist(x,y)$.
See \cite{AmbrosioGigliSavare14} (dealing with finite reference measures),
\cite{AmbrosioGigliMondinoRajala} (for the $\sigma$-finite case) and \cite{AmbrosioGigliSavare15} for various characterizations
of the class of $RCD(K,\infty)$ spaces. We quote here a few results, which essentially derive from the identification of $h_t$
as the gradient flow of ${\rm Ent}$ w.r.t. the Wasserstein distance and the contractivity properties with respect to that distance.

It is proved in \cite{AmbrosioGigliSavare14} that the formula
$$
h_t g(x):=\int_X g(y) d\tilde h_t\delta_x(y)\qquad x\in X,\, t\geq 0
$$
where $\tilde h_t$ is the dual $K$-contractive semigroup acting on $\Prob_2(X)$,
provides a pointwise version of the semigroup on $L^2\cap L^\infty(X,\meas)$
with better continuity properties, recalled among other things in the next proposition.
Notice also that the formula
$$
\tilde h_t\mu:=\int \tilde h_t\delta_x\di\mu(x)
$$
provides a canonical extension of $\tilde h_t$ to the whole of $\Prob(X)$, used in Proposition~\ref{prop:char_bv}.

\begin{proposition} [Regularizing properties of $h_t$] \label{prop:reguht} Let $(X,\dist,\meas)$ be a $RCD(K,\infty)$ metric measure space. Then,
any $f\in H^{1,2}(X,\dist,\meas)$ with $|\nabla f|\in L^\infty(X,\meas)$ has a Lipschitz representative 
$\tilde f$, with $\Lip(\tilde f)=\||\nabla f|\|_{L^\infty(X,\meas)}$ and 
the following properties hold for all $t>0$:
\begin{itemize}
\item[(a)] if $f\in L^2\cap L^\infty(X,\meas)$ one has $h_t f\in\Lipb(X)\cap H^{1,2}(X,\dist,\meas)$ with 
\begin{equation}\label{eq:regu1}
|\nabla h_t f|={\rm lip}(h_t f)\quad\text{$\meas$-a.e.},\qquad
\Lip(h_t f)\leq \frac{1}{\sqrt{2{\sf I}_{2K}(t)}}\|f\|_{L^\infty(X,\meas)};
\end{equation}
\item[(b)] for all $f\in H^{1,2}(X,\dist,\meas)$ with $|\nabla f|\in L^\infty(X,\meas)$ the Bakry-\'Emery condition holds in the form
\begin{equation}\label{eq:BE}
{\rm Lip}_a(h_t f,x)\leq e^{-Kt}h_t|\nabla f|(x)\qquad\forall x\in X;
\end{equation}
\item[(c)] if $\mu\in\Prob_2(X)$, then $\tilde{h}_t\mu=f_t\meas$, with
$$
\int_X f_t\log f_t\di\meas\leq \frac{1}{2{\sf I}_{2K}(t)}\biggl(r^2+\int_X\dist^2(x,\bar x)\di\mu(x)\biggr)
-\log\bigl(\meas (B_r(\bar x))\bigr)
$$
for all $\bar x\in X$ and $r>0$.
\end{itemize}
\end{proposition}
\begin{proof} (a) is proved in \cite{AmbrosioGigliSavare14,AmbrosioGigliSavare15}, (b) in \cite{Savare}. The inequality (c) follows by
Wang's log-Harnack inequality, see \cite[Theorem~4.8]{AmbrosioGigliSavare15} for a proof in the $RCD(K,\infty)$ context.
\end{proof}

In $RCD(K,\infty)$ spaces we have also a useful formula to represent the functional $\int_X|\nabla f|\di\meas$.

\begin{proposition} \label{prop:newrepGamma} For all $f\in H^{1,2}(X,\dist,\meas)$ one has that $|\nabla f|$ is the essential supremum of the family
$\langle\nabla f,\nabla v\rangle$ as $v$ runs in the family of $1$-Lipschitz functions in $H^{1,2}(X,\dist,\meas)$. Moreover, for all
$g:X\to [0,\infty)$ lower semicontinuous, one has
\begin{equation}\label{eq:supformula}
\int_X|\nabla f|g\di\meas=\sup\sum_k\int_X\langle \nabla f,\nabla v_k\rangle w_k\di\meas
\end{equation}
where the supremum runs among all finite collections of $1$-Lipschitz functions $v_k\in H^{1,2}(X,\dist,\meas)$ and 
all $w_k\in\Cbs(X)$ with $\sum_k|w_k|\leq g$.
\end{proposition}
\begin{proof} The proof of the representation of $|\nabla f|$ as essential supremum has been achieved
in \cite[Lemma~9.2]{AmbrosioTrevisan14}. We sketch the argument: denoting
by $M$ the essential supremum in the statement, one has obviously the inequalities $M\leq |\nabla f|$ $\meas$-a.e. and
$|\langle\nabla f,\nabla v\rangle|\leq M\Lip(v)$ $\meas$-a.e. for all $v\in H^{1,2}(X,\dist,\meas)$ Lipschitz and bounded. 
By localization, this last inequality is improved to $|\langle \nabla f,\nabla v\rangle|\leq M{\rm Lip}_a(v)$ $\meas$-a.e. for all
$v\in H^{1,2}(X,\dist,\meas)$ Lipschitz and bounded and then a density argument provides the inequality $|\langle\nabla f,\nabla v\rangle|\leq M|\nabla v|$
for all $v\in H^{1,2}(X,\dist,\meas)$ Lipschitz and bounded, which leads to $|\nabla f|\leq M$ choosing $v=f$.

In order to prove \eqref{eq:supformula} we remark that the representation of $|\nabla f|$ as essential supremum yields
$$
\int_Xg|\nabla f|\di\meas=\sup\sum_kc_k\int_{B_k}\langle \nabla f,\nabla v_k\rangle \di\meas
$$
where the supremum runs among all finite Borel partitions $B_k$ of $X$, constants $c_k\leq\inf_{B_k} g$ and all choices of bounded $1$-Lipschitz functions 
$v_k\in H^{1,2}(X,\dist,\meas)$. By inner regularity, the supremum is unchanged if we replace the Borel partitions by finite families of pairwise
disjoint compact sets $K_k$. In turn, these families can be approximated by functions $w_k\in\Cbs(X)$ with $\sum_k|w_k|\leq g$.
\end{proof}

Now we recall three useful functional inequalities available in $RCD(K,\infty)$ spaces.

\begin{proposition} \label{prop:Ledoux} If $(X,\dist,\meas)$ is a $RCD(K,\infty)$ metric measure space, for all $f\in\Lip_\bs(X)$ one has
\begin{equation}\label{eq:Ledoux}
\int_X |h_t f-f|\di\meas\leq c(t,K)\int_X|\nabla f|\di\meas
\end{equation}
with $c(t,K)\sim \sqrt{t}$ as $t\downarrow 0$.
\end{proposition}
\begin{proof} Fix $g\in L^\infty(X,\meas)$ with $\|g\|_{L^\infty(X,\meas)}\leq 1$ and let us estimate the derivative
of $t\mapsto\int_X gh_t f\di\meas$:
\begin{eqnarray*}
\biggl|\int_X g\Delta h_t f\di\meas\biggr|&=&
\biggl|\int_X gh_{t/2}\Delta h_{t/2} f\di\meas\biggr|=
\biggl|\int_X h_{t/2} g\Delta h_{t/2} f\di\meas\biggr|\\
&=&\biggl|\int_X\langle\nabla h_{t/2} g,\nabla h_{t/2} f\rangle\di\meas\biggr|\leq
  \frac{1}{\sqrt{2{\sf I}_{2K}(t/2)}}\int_X|\nabla h_{t/2}f|\di\meas\\&\leq&
  \frac{e^{-Kt/2}}{\sqrt{2{\sf I}_{2K}(t/2)}}\int_X|\nabla f|\di\meas.
\end{eqnarray*}
By integration, and then taking the supremum w.r.t. $g$, we get \eqref{eq:Ledoux}.
\end{proof}

%

When the space has finite diameter and $K\leq 0$ 
we will also use, as a replacement of the isoperimetric inequality (presently known in the $RCD(K,\infty)$ setting only when $K>0$), 
the following inequality, which is an easy consequence of Proposition~\ref{prop:reguht}(c).

\begin{proposition} \label{prop:quasiiso} If $(X,\dist,\meas)$ is a $RCD(K,\infty)$ metric measure space with $\meas(X)=1$, and 
if $D=\supp\meas$ is finite, for all $\epsilon>0$ we can find $M=M(\epsilon,D,K)\geq 1$ such that
$$
\int_{\{f\geq M\int_X f\di\meas\}} f\di\meas\leq \epsilon\biggl(\int_X f\di\meas+\int_X|\nabla f|\di\meas\biggr).
$$
for all $f\in\Lipb(X)$ nonnegative.
\end{proposition}
\begin{proof} The standard entropy inequality 
$$
\int_A g\di\meas\log\biggl(\frac{1}{\meas (A)}\int_Ag\di\meas\biggr)\leq
\int_A g\log g\di\meas\leq \int_Xg\log g\di\meas+\frac 1e \meas(X\setminus A)
$$
provides a modulus of continuity $\omega_E$, depending only on $E\geq 0$, such that $g$ nonnegative and 
$\int_X g\log g\di\meas\leq E$ imply $\int_A g\di\meas\leq\omega_E(\meas(A))$.

Assume first $\int_Xf\di \meas=1$ and let $M> 0$. For all $t>0$
we apply Proposition~\ref{prop:Ledoux} and Proposition~\ref{prop:reguht}(c) with $r=D$ to get
\begin{eqnarray}\label{eq:sufficient}
\int_{\{f\geq M\}} f\di\meas&\leq&\int_{\{f\geq M\}}h_tf\di\meas+\int_X|h_tf-f|\di\meas\\&\leq&
\omega_{E_t}(\frac{1}{M})+c(K,t)\int_X|\nabla f|\di\meas\nonumber
\end{eqnarray}
with 
$$
E_t=\frac{D^2}{{\sf I}_{2K}(t)}\geq\int_X h_tf\log h_t f\di\meas.
$$
By a scaling argument, the inequality \eqref{eq:sufficient} implies
$$
\int_{\{f\geq M\int_X f\di\meas\}} f\di\meas\leq\omega_{E_t}(\frac 1M)\int_X f\di\meas
+c(K,t)\int_X|\nabla f|\di\meas\qquad\forall t,\,M>0.
$$
Then, given $\epsilon>0$ we choose first $t>0$ sufficiently small such that $c(t,K)<\epsilon$ and then
$M$ sufficiently large to conclude.
\end{proof}

Finally, we close this section by reminding higher order properties, strongly inspired by Bakry's calculus, which played a fundamental
role in the recent developments of the theory. 

\begin{proposition} \label{prop:Bakry} Let $(X,\dist,\meas)$ be a $RCD(K,\infty)$ space. Then
\begin{equation}\label{eq:bound_on_Gamma1}
\||\nabla f|\|_{L^4(X,\meas)}\leq c\|f\|_\infty\|(\Delta-K^-I)f\|_{L^2(X,\meas)}
\end{equation}
for all $f\in L^\infty(X,\meas)\cap D(\Delta)$, and
\begin{equation}\label{eq:bound_on_Gamma2}
\|\nabla |\nabla g|^2\|^2_{L^2(X,\meas)}\leq -\int_X\bigl(2K|\nabla g|^4+2|\nabla g|^2\langle \nabla g,\nabla \Delta g\rangle)\di\meas
\end{equation}
for all $g\in H^{1,2}(X,\dist,\meas)\cap\Lipb(X)\cap D(\Delta)$ with $\Delta g\in H^{1,2}(X,\dist,\meas)$.
\end{proposition}
\begin{proof} See \cite[Theorem~3.1]{AmbrosioMondinoSavare16} for \eqref{eq:bound_on_Gamma1}, \cite[Section~3]{Savare} for \eqref{eq:bound_on_Gamma2}.
\end{proof}

\section{Local convergence of gradients under Mosco convergence}\label{sec:5}

The main goal of this section is to localize the Mosco convergence result of \cite{GigliMondinoSavare13}, proving convergence
results for $\langle\nabla u_i,\nabla v_i\rangle_i$ to $\langle\nabla u,\nabla v\rangle$ when $u_i$ are strongly convergent in $H^{1,2}$
to $u$ and $v_i$ are weakly convergent in $H^{1,2}$ to $v$. When both sequences are strongly convergent,
we obtain at least the weak convergence as measures. Besides Theorem~\ref{thm:gms13_flow} borrowed from
\cite{GigliMondinoSavare13}, the main tool is the convergence results (in the more general context of derivations)
of \cite{AmbrosioStraTrevisan}, see Theorem~\ref{thm:strong-convergence-gradients}.

\begin{definition}[Mosco convergence]\label{def:Mosco}
We say that the Cheeger energies $\Ch^i:=\Ch_{\meas_i}$ Mosco converge to $\Ch$ if both the 
following conditions hold:
\begin{itemize}
\item[(a)] (\emph{Weak-$\liminf$}). For every $f_i\in L^2(X,\meas_i)$ $L^2$-weakly converging to $f\in L^2(X,\meas)$, one has
\[ \Ch(f)\le \liminf_{i\to\infty} \Ch^i(f_i).\]
\item[(b)] (\emph{Strong-$\limsup$}). For every $f \in L^2(X,\meas)$ there exist $f_i\in L^2(X,\meas_i)$, $L^2$-strongly converging to $f$ with
\begin{equation}\label{eq:optimal_Chee} \Ch(f)=\lim_{i\to \infty} \Ch^i(f_i).\end{equation}
\end{itemize}
\end{definition}

One of the main results of \cite{GigliMondinoSavare13} is that Mosco convergence holds if $(X,\dist,\meas_i)$ are $RCD(K,\infty)$ spaces with 
\begin{equation}\label{eq:uniform_growth}
\meas_i(B_r(\bar x))\leq c_1e^{c_2r^2}\qquad\forall r>0,\,\,\forall i
\end{equation} 
for some $\bar x\in X$ and $c_1,\,c_2>0$. Notice that this result holds even in the larger class of $CD(K,\infty)$ spaces and that
the uniform growth condition \eqref{eq:uniform_growth}, that we prefer to emphasize, is actually a consequence of the local weak convergence
of $\meas_i$ to $\meas$ and of the uniform lower bound on Ricci curvature (see Remark~\ref{rem:Grygorian}).

Next, we define in a natural way, following \cite{GigliMondinoSavare13}, weak and strong convergence in the Sobolev space $H^{1,2}$,
with a variable reference measure.

\begin{definition}[Convergence in the Sobolev spaces]
We say that $f_i\in H^{1,2}(X,\dist,\meas_i)$ are weakly convergent in $H^{1,2}$ to 
$f\in H^{1,2}(X,\dist,\meas)$ if $f_i$ are $L^2$-weakly convergent to $f$ 
and $\sup_i\Ch^i(f_i)$ is finite. Strong convergence
in $H^{1,2}$ is defined by requiring $L^2$-strong convergence of the functions, and that $\Ch(f)=\lim_i\Ch^i(f_i)$. 
\end{definition}

Notice that the sequence $f_i=h$, with $h\in\Lip_\bs(X)$ fixed, need not be strongly convergent in $H^{1,2}$, as the
following simple example taken from \cite{AmbrosioStraTrevisan} shows. The reason is that this sequence should
not be considered as a constant one, since the supports of $\meas_i$ can well be pairwise disjoint.

\begin{example}
Take $X=\setR^2$ endowed with the Euclidean distance, $f(x_1,x_2)=x_2$ and let 
$$
\meas_i=i\leb^2\res \bigl([0,1]\times [0,\tfrac 1i]\bigr),\qquad
\meas=\haus^1\res [0,1]\times\{0\}.
$$
Then, it is easily seen that $|\nabla f|_i=1$, while $|\nabla f|=0$.
\end{example}

It is immediate to check that weak convergence in $H^{1,2}$ is stable under finite sums; it follows from \eqref{eq:continuitymean} below that the
same holds for strong convergence in $H^{1,2}$. Also, Theorem~\ref{thm:gms13} below (borrowed from 
\cite{GigliMondinoSavare13}) yields that weakly convergent sequences are also
$L^2_\loc$-strongly convergent, and provides conditions under which this can be improved to $L^2$-strong convergence.

\begin{theorem}[Mosco convergence under uniform Ricci bounds]\label{thm:gms13_flow}
If $(X,\dist,\meas_i)$ are $RCD(K,\infty)$ spaces satisfying \eqref{eq:uniform_growth}, then $\Ch^i$ Mosco converge to $\Ch$. In addition
\begin{equation}\label{eq:continuitymean}
\lim_{i\to\infty}\int_X\langle\nabla v_i,\nabla w_i\rangle_i\di\meas_i=\int_X\langle\nabla v,\nabla w\rangle\di\meas,
\end{equation}
whenever $(v_i)$ strongly converge in $H^{1,2}$ to $v$ and $(u_i)$ weakly converge in $H^{1,2}$ to $u$ and
the heat flows $h^i$ relative to $(X,\dist,\meas_i)$ converge to the heat flow $h$ relative to $(X,\dist,\meas)$ in the following sense: 
\begin{equation}\label{thm:gms13_flow2}
\text{$\forall t\geq 0$, $h^i_tf_i$ $L^2$-strongly converge to $h_t f$ whenever $f_i$ $L^2$-strongly converge to $f$.}
\end{equation}
\end{theorem}
\begin{proof} See \cite[Theorem~6.8]{GigliMondinoSavare13} for the Mosco convergence and \cite[Theorem~6.11]{GigliMondinoSavare13}
for the $L^2$-strong convergence of $h^i_t f_i$ to $h_t f$. The proof of \eqref{eq:continuitymean} is elementary:
since $v_i + t w_i$ weakly converge in $H^{1,2}$ to $v+t w$ for all $t>0$, by Mosco convergence we have
\begin{equation*}
\begin{split}
\Ch(v) + 2t \int_X \langle\nabla v,\nabla w\rangle \di \meas + t^2 \Ch(w) & = \Ch(v+ t w) \le \liminf_{i \to \infty}  \Ch^i(v_i+ t w_i) \\
& = \liminf_{i \to \infty}  \Ch^i(v_i) + 2t \int_X \langle\nabla v_i,\nabla w_i\rangle_i \di \meas_i + t^2 \Ch^i(g_i)\\
& \le  \Ch(v) +  2 t \liminf_{i\to\infty}\int_X \langle\nabla v_i,\nabla w_i\rangle_i \di \meas + t^2 \limsup_{i\to\infty} \Ch^i(w_i).
\end{split}
\end{equation*}
Since $\sup_i \Ch^i(w_i)$ is finite, we may let $t \downarrow 0$ to deduce the $\liminf$ inequality;
replacing $w$ by $-w$ gives \eqref{eq:continuitymean}.
 \end{proof}

In the following corollary we prove standard consequences of the Mosco convergence of Theorem~\ref{thm:gms13_flow}, which
refine \eqref{thm:gms13_flow2}
 (see also \cite[Corollary~6.10]{GigliMondinoSavare13} for a discrete counterpart of this result, involving the resolvents).

\begin{corollary} \label{cor:strocoh} Under the same assumptions of Theorem~\ref{thm:gms13_flow}, one has
\begin{itemize}
\item[(a)] if $f_i\in H^{1,2}(X,\dist,\meas_i)$, $f_i\in D(\Delta_i)$ $L^2$-strongly converge to $f$ 
and $\Delta_i f_i$ is uniformly bounded in $L^2$, then $f\in D(\Delta)$, $\Delta_i f_i$ $L^2$ weakly converge to
$\Delta f$ and $f_i$ strongly converge in $H^{1,2}$ to $f$;
\item[(b)] for all $t>0$, $h^i_tf_i$ strongly converge in $H^{1,2}$ to $h_t f$
whenever $f_i$ $L^2$-strongly converge to $f$.
\end{itemize}
\end{corollary}
\begin{proof}  (a) Using the integration by parts formula we see that $f_i$ is weakly convergent in $H^{1,2}$.
Let $\chi\in H^{1,2}(X,\dist,\meas)$ and let $\chi_i\in H^{1,2}(X,\dist,\meas_i)$ be strongly convergent to $\chi$
in $H^{1,2}$. Let $g$ be a $L^2$-weak limit point of $\Delta_i f_i$ as $i\to\infty$, so that
\eqref{eq:continuitymean} gives (along a subsequence, that for simplicity we do not denote explicitly)
$$
\int_X g\chi \di\meas=\lim_{i\to\infty}\int_X \chi_i\Delta_i f_i\di\meas_i=
-\lim_{i\to\infty}\int_X\langle\nabla\chi_i,\nabla f_i\rangle_i\di\meas_i
=-\int_X\langle\nabla\chi,\nabla f\rangle\di\meas.
$$
This proves that $f\in D(\Delta)$ and $g=\Delta f$, so that compactness gives that $\Delta_i f_i$ $L^2$-weakly
converge to $\Delta f$. We can pass to the limit in the integration by parts formula
$\int_X|\nabla f_i|_i^2\di\meas_i=-\int_X f_i\Delta_i f_i\di\meas_i$
to prove the strong $H^{1,2}$ convergence of $f_i$ to $f$.

Now, we can prove (b). From \eqref{eq:Brezis2} we know that $\Delta_i h^i_t f_i$ is bounded in
$L^2$ for all $t>0$, hence (a) provides the strong convergence in $H^{1,2}$ of $h^i_t f_i$ to $h_t f$. 
\end{proof}

In order to localize the previous results (see in particular \eqref{eq:continuitymean})
we shall use the next theorem, proved in \cite[Theorem~5.3]{AmbrosioStraTrevisan}.
It shows that any sequence $(f_i)$ strongly convergent in $H^{1,2}$ to $f$  
induces gradient derivations which are strongly converging to the gradient derivation of the limit function, using as class of test functions the
family $h_{\setQ^+}\Algebra_\bs$ defined below
\begin{equation}\label{eq:def_Cinfty}
h_{\setQ_+}\Algebra_\bs:=\left\{h_sf:\ f\in\Algebra_\bs,,\,\,s\in\setQ_+\right\}\subset\Lipb(X).
\end{equation}
Notice that $h_{\setQ_+}\Algebra_\bs$  depends only on the limit metric measure structure, and it
is dense in $H^{1,2}(X,\dist,\meas)$, see \cite[Theorem~B.1]{AmbrosioStraTrevisan}.
Notice also that, since $\supp\meas$ can well be a strict subset of $X$, the $\Lipb(X)$ extension of $f\in h_{\setQ_+}\Algebra_\bs$
is not necessarily unique, and therefore $\langle\nabla v,\nabla f\rangle_i$ might depend on this extension, when $v\in H^{1,2}(X,\dist,\meas_i)$
(while $\langle\nabla v,\nabla f\rangle$ does not, for $v\in H^{1,2}(X,\dist,\meas)$). 
Nevertheless, the following convergence theorem is independent of the extension. 

\begin{theorem}[Strong convergence of gradients]\label{thm:strong-convergence-gradients}
Assume that $(X,\dist,\meas)$ is a $RCD(K,\infty)$ metric measure space, 
that $\Ch^i$ are quadratic and that Mosco converge to $\Ch$.
Let $v_i\in H^{1,2}(X,\dist,\meas_i)$ be strongly convergent in $H^{1,2}$ to $v\in H^{1,2}(X,\dist,\meas)$. \\
Then, for all $f\in h_{\setQ^+}\Algebra_\bs$, $\langle\nabla v_i,\nabla f\rangle_i$ 
$L^2$-strongly converge to $\langle\nabla v,\nabla f\rangle$. 
\end{theorem}

\begin{theorem} [Continuity of the gradient operators] \label{thm:cont_reco} 
Assume that $(X,\dist,\meas_i)$ are $RCD(K,\infty)$ metric measure spaces, 
let $v\in H^{1,2}(X,\dist,\meas)$ and let $v_i\in H^{1,2}(X,\dist,\meas_i)$
be strongly convergent in $H^{1,2}$ to $v$. Then:
\begin{itemize} 
\item[(a)] the following tightness on bounded sets holds:
\begin{equation}\label{eq:tightbounded}
\lim_{R\to\infty}\limsup_{i\to\infty}\int_{X\setminus B_R(\bar x)}|\nabla v_i|_i^2\di\meas_i=0.
\end{equation}
\item[(b)] If $w_i$ weakly converge to $w$ in $H^{1,2}$,
the measures $\langle\nabla v_i,\nabla w_i\rangle_i\meas_i$ weakly converge in duality with
$h_{\setQ_+}\Algebra_\bs$ to $\langle\nabla v,\nabla w\rangle\meas$ and, if $\langle \nabla v_i,\nabla w_i\rangle_i$ is bounded
in $L^p$ for some $p\in (1,\infty)$, also weakly in $L^p$.
\item[(c)] If $w_i$ strongly converge to $w$ in $H^{1,2}$, then $\langle \nabla v_i,\nabla w_i\rangle_i$
$L^1$-strongly converge to $\langle\nabla v,\nabla w\rangle$. 
\end{itemize} 
\end{theorem}
\begin{proof} (a) In order to prove \eqref{eq:tightbounded} we choose $\chi_R:X\to [0,1]$
$1/R$-Lipschitz with $\chi_R\equiv 0$ on $B_R(\bar x)$, $\chi_R\equiv 1$ on $X\setminus B_{2R}(\bar x)$ and notice
that the Leibniz rule gives
$$
\int_X|\nabla v_i|_i^2\chi_R\di\meas_i=\int_X\langle\nabla v_i,\nabla (v_i\chi_R)\rangle_i\di\meas_i-
\int_X\langle\nabla v_i,\nabla\chi_R\rangle v_i\di\meas_i,
$$
so that we can use \eqref{eq:continuitymean} to get
$$
\limsup_{i\to\infty}\int_X|\nabla v_i|_i^2\chi_R\di\meas_i\leq
\int_X\langle\nabla v,\nabla (v\chi_R)\rangle \di\meas+
\frac {1}{R}\biggl(\int_X|\nabla v|^2\di\meas\biggr)^{1/2}\|v\|_{L^2(X,\meas)}.
$$
Using the Leibniz rule once more we get
$$
\limsup_{i\to\infty}\int_X|\nabla v_i|_i^2\chi_R\di\meas_i\leq\int_X|\nabla v|^2\chi_R\di\meas+
\frac {2}{R}\biggl(\int_X|\nabla v|^2\di\meas\biggr)^{1/2}\|v\|_{L^2(X,\meas)},
$$
which gives \eqref{eq:tightbounded}.

Let us now prove (b). Let $f\in h_{\setQ_+}\Algebra_\bs$. Using the Leibniz rule we can write
$$
\int_X  \langle\nabla v_i,\nabla w_i\rangle_i f\di\meas_i=-\int_X  \langle\nabla v_i,\nabla f\rangle_i w_i\di\meas_i+
\int_X  \langle\nabla v_i,\nabla (w_if)\rangle_i \di\meas_i
$$
and use \eqref{eq:continuitymean} together with the $L^2$-strong convergence of $\langle \nabla v_i,\nabla f\rangle_i$
to $\langle\nabla v,\nabla f\rangle$, ensured by Theorem~\ref{thm:strong-convergence-gradients}, to conclude the
weak convergence in duality with $h_{\setQ_+}\Algebra_\bs$ of $\langle\nabla v_i,\nabla w_i\rangle_i\meas_i$.
Assuming in addition that $\langle\nabla v_i,\nabla w_i\rangle_i$ satisfy a uniform $L^p$ bound for some $p>1$, 
let $\xi\in L^p(X,\meas)$ be the $L^p$-weak limit of  a subsequence (not relabelled for simplicity of notation). Then,
\eqref{eq:tightbounded} gives
$$
\limsup_{i\to\infty}\biggl|\int_X \langle\nabla v_i,\nabla w_i\rangle_i\phi\psi_R\di\meas_i-
\int_X \langle\nabla v_i,\nabla w_i\rangle_i\phi\di\meas_i\biggr|=o(R)
$$
with $\phi\in h_{\setQ_+}\Algebra_\bs$, $\psi_R=1-\chi_R\in\Lip_\bs(X)$, $\chi_R$ chosen as in the proof of (a), hence
we can pass to the limit as $i\to\infty$ to get
$$
\biggl|\int_X \xi\phi\psi_R\di\meas-
\int_X \langle\nabla v,\nabla w\rangle\phi\di\meas\biggr|=o(R).
$$
Since $h_{\setQ_+}\Algebra_\bs$ is dense in $L^q(X,\meas)$, with $q$ dual exponent of $p$, we can pass
to the limit as $R\to\infty$ and use the arbitrariness of $\phi$ to obtain that $\xi=\langle\nabla v,\nabla w\rangle$.

In order to prove (c), by polarization and the linearity of $L^1$-strong convergence it is not restrictive to assume $v_i=w_i$. 
It is then sufficient to apply \eqref{eq:lscc1} of Lemma~\ref{lem:old_sci} below (whose proof uses only (a), (b) of this proposition) 
to obtain the inequality $\liminf_i\int_A|\nabla f_i|_i\di\meas_i\geq\int_A|\nabla f|\di\meas$
on any open set $A\subset X$. Assume that $\xi\in L^2(X,\meas)$ is a $L^2$-weak limit point of $|\nabla f_i|_i$; from the
liminf inequality we get $\int_A\xi\di\meas\geq\int_A|\nabla f|\di\meas$ for any open set $A$ with $\meas(\partial A)=0$.
A standard approximation then gives $\xi\geq |\nabla f|$ $\meas$-a.e. in $X$. Since the $H^{1,2}$ strong convergence
gives
$$
\limsup_{i\to\infty}\int_X|\nabla f_i|_i^2\di\meas_i\leq\int_X|\nabla f|^2\di\meas\leq
\int_X\xi^2\di\meas,
$$
we obtain the $L^2$-strong convergence of $|\nabla f_i|_i$. Combinig the inequality above with
$\liminf_i\| |\nabla f_i|_i\|_{L^2(X,\meas_i)}\geq\|\xi\|_{L^2(X,\meas)}$ we obtain that $\xi=|\nabla f|$.
 \end{proof}

\begin{lemma} \label{lem:old_sci} If $f_i\in H^{1,2}(X,\dist,\meas_i)$ weakly converge in $H^{1,2}$ to $f$, then
\begin{equation}\label{eq:lscc1}
\liminf_{i\to\infty}\int_X g|\nabla f_i|_i\di\meas_i\geq\int_X g|\nabla f|\di\meas
\end{equation}
for any lower semicontinuous $g:X\to [0,\infty]$ and then
\begin{equation}\label{eq:lscc2}
\liminf_{i\to\infty}\int_A|\nabla f_i|_i^2\di\meas_i\geq\int_A|\nabla f|^2\di\meas
\end{equation}
for any open set $A\subset X$.
\end{lemma}
\begin{proof} Since truncation preserves $L^2_\loc$-strong convergence and uniform $L^2$ bounds, 
by a truncation argument, in the proof of \eqref{eq:lscc1} we can assume with no loss of generality
that $f_i$ are uniformly bounded. Since any lower semicontinuous function is the monotone limit of a sequence of Lipschitz
functions with bounded support, we also assume $g\in\Lip_\bs(X)$. Also, taking into account the inequality 
$|\nabla h^i_t f_i|_i\leq e^{-Kt}h^i_t|\nabla f|_i$, we can estimate
\begin{eqnarray*}
\liminf_{i\to\infty}\int_X g|\nabla f_i|_i\di\meas_i&\geq&
\liminf_{i\to\infty}\int_X h_t^ig |\nabla f_i|_i\di\meas_i-
\limsup_{i\to\infty}\int_X |h_t^ig-g| |\nabla f_i|_i\di\meas_i\\
&\geq&
e^{Kt}\liminf_{i\to\infty}\int_X g |\nabla h_t^i f_i|_i\di\meas_i
-C\limsup_{i\to\infty}\|h_t^ig-g\|_{L^2(X,\meas_i)},
\end{eqnarray*}
with $C=\sup_i (2\Ch^i(f_i))^{1/2}$. Since \eqref{eq:Ledoux} gives
$$
\lim_{t\to 0}\limsup_{i\to\infty}\int_X |h^i_t g-g|^2\di\meas_i=0,
$$
this means that as soon as we have the liminf inequality for $h^i_tf_i$, $h_t f$ for all $t>0$, we have it for $f_i$, $f$.

Hence,  possibly replacing $f_i$ by $h^i_t f_i$ we see thanks to \eqref{eq:regu1} that we can assume 
with no loss of generality that $f_i$ are uniformly Lipschitz. Under this assumption, we first 
prove \eqref{eq:lscc1} in the case when $g=\chi_A$ is the characteristic function of an open set $A\subset X$, 
we fix finitely many $v_k\in H^{1,2}(X,\dist,\meas)$ with $\Lip(v_k)\leq 1$, as well as finitely many
$w_k\in\Cbs(X)$ with $\supp w_k\subset A$ and $\sum_k|w_k|\leq 1$. Let us also  fix $v_{k,i}$ strongly convergent in $H^{1,2}$ to $v_k$.  
Now, notice that
\begin{equation}\label{eq:exp1bis}
\lim_{i\to\infty}\int_X\langle\nabla f_i,\nabla v_{k,i}\rangle_i w_k\di\meas_i
=\int_X\langle\nabla f,\nabla v_k\rangle w_k \di\meas\qquad\forall k.
\end{equation}
Indeed, \eqref{eq:exp1bis} follows at once from the weak $L^2$ convergence
of $\langle\nabla f_i,\nabla v_{k,i}\rangle_i$ to $\langle\nabla f,\nabla v_k\rangle$ provided by Theorem~\ref{thm:cont_reco}(b).
Adding with respect to $k$, since $\Lip(v_{k,i})\leq 1$ and $\sum_k |w_k|\leq \chi_A$, from \eqref{eq:supformula} with
$g\equiv \chi_A$ we get \eqref{eq:lscc1}.

For general $g$ we use the formula
$$
\int g h\di\mu=\int_0^\infty\int_{\{g>t\}}h\di\mu\di t
$$
(with $\mu=\meas_i$ and $\mu=\meas$) and Fatou's lemma.

The proof of \eqref{eq:lscc2} is a direct consequence of the elementary identity 
$$
\int_A u^2\di\meas=\sup\left\{\sum_k\meas(A_k)^{-1}\biggl(\int_{A_k}|u|\di\meas\biggr)^2\right\},
$$
where the supremum runs among the finite disjoint families of open subsets $A_k$ of $A$ with $\meas(A_k)>0$, 
of \eqref{eq:lscc1} and of the superadditivity of the $\liminf$ operator.
\end{proof}

\section{$BV$ functions and their stability}\label{sec:6}

In this section we first recall basic facts about $BV$ functions in metric measure spaces. The most important result of
this section, estabilished in Theorem~\ref{thm:main}, is the extension of a well-known fact, namely the stability of $BV$ functions 
under $L^1$-strong convergence, to the case when even the family of spaces is variable.

\begin{definition}[The class $BV(X,\dist,\meas)$ and $|\rmD f|(X)$]
We say that $f\in L^1(X,\meas)$ belongs to $BV(X,\dist,\meas)$ if there exist functions $f_n\in L^1(X,\meas)\cap \Lipb(X)$ 
convergent to $f$ in $L^1(X,\meas)$ with
\begin{equation}\label{eq:liminf}
L:=\liminf_{n\to\infty}\int_X{\rm lip}(f_n)\di\meas<\infty,
\end{equation}
where ${\rm lip}(g)$ denotes the local Lipschitz constant of $g$, see \eqref{eq:localsl}.
If $f\in BV(X,\dist,\meas)$, the optimal $L$ in \eqref{eq:liminf} (i.e. the inf of $\liminf$) is called total variation of $f$ and
denoted by $|\rmD f|(X)$. By convention, we put $|\rmD f|(X)=\infty$ if $f\in L^1\setminus BV(X,\dist,\meas)$.
\end{definition}

It is immediate to check from the definition of total variation that for $\phi\circ f\in BV(X,\dist,\meas)$ for all
$f\in BV(X,\dist,\meas)$ and all $\phi:\setR\to\setR$ 1-Lipschitz with $\phi(0)=0$, with 
\begin{equation}\label{eq:monotot}
|\rmD (\phi\circ f)|(X)\leq |\rmD f|(X).
\end{equation}
In addition, the very definition of $|\rmD f|(X)$ provides the lower semicontinuity property
$$
|\rmD f|(X)\leq\liminf_{n\to\infty}|\rmD f_n|(X)
\qquad\text{whenever $f_n\to f$ in $L^1(X,\dist,\meas)$.}
$$
Still using the lower semicontinuity, arguing as in \cite{Miranda}, one can prove the coarea formula
\begin{equation}\label{eq:coarea}
|\rmD f|(X)=\int_0^\infty|\rmD \chi_{\{f>t\}}|(X)\di t
\qquad\forall f\in L^1(X,\meas),\,\,f\geq 0.
\end{equation}

In the following proposition, whose proof was suggested to the first author by S. Di Marino, we provide a useful equivalent
representation of $|\rmD f|(X)$.

\begin{proposition}\label{prop:better_var}
For all $f\in L^1(X,\meas)$ one has
$$
|\rmD f|(X)=\inf\liminf_{n\to\infty}\int_X{\rm Lip}_a(f_n)\di\meas,
$$
where the infimum runs amont all $f_n\in\Lip_\bs(X)$ 
convergent to $f$ in $L^1(X,\meas)$.
\end{proposition}
\begin{proof} By a diagonal argument it is sufficient, for any $f\in\Lipb(X)$ with ${\rm lip}(f)\in L^1(X,\meas)$, to 
find $f_n\in\Lip_\bs(X)$ convergent to $f$ in $L^1(X,\meas)$ with ${\rm Lip}_a(f_n)\to g$ in $L^1(X,\meas)$ and
$g\leq{\rm lip}(f)$ $\meas$-a.e. in $X$. By a further diagonal argument, it is sufficient to find $f_n$ when $f\in\Lip_\bs(X)$.
Under this assumption, we know by Theorem~\ref{tapprox} that there exist $f_n\in\Lipb(X)$ satisfying $f_n\to f$ in $L^2(X,\meas)$ with
${\rm Lip}_a(f_n)\to |\nabla f|$ in $L^2(X,\meas)$. Since $f$ has bounded support, also $f_n$ can be taken with equibounded
support, hence both convergences occur in $L^1(X,\meas)$. Since $|\nabla f|\leq {\rm lip}(f)$ $\meas$-a.e., we are done.
\end{proof}

In the following proposition we list more properties of $BV$ functions in $RCD(K,\infty)$ spaces.

\begin{proposition}\label{prop:char_bv}
Let $(X,\dist,\meas)$ be a $RCD(K,\infty)$ space. Then, the following properties hold:
\begin{itemize}
\item[(a)] if $f\in \Lipb(X)\cap L^1(X,\meas)\cap H^{1,2}(X,\dist,\meas)$ one has 
\begin{equation}\label{eq:regu2}
|\rmD f|(X)=\int_X |\nabla f|\di\meas;
\end{equation}
\item[(b)] if $f\in BV(X,\dist,\meas)$ one has
\begin{equation}\label{eq:BVcontraction}
|Dh_t f|(X)\leq e^{-Kt}|\rmD f|(X);
\end{equation} 
\item[(c)] for all $f\in BV(X,\dist,\meas)$ one has
\begin{equation}\label{eq:Ledoux1}
\int_X |P_t f-f|\di\meas\leq c(t,K)|\rmD f|(X)
\end{equation}
with $c(t,K)\sim \sqrt{t}$ as $t\downarrow 0$.
\end{itemize}
\end{proposition}
\begin{proof} (a) Let $f\in\Lipb(X)\cap L^1(X,\meas)\cap H^{1,2}(X,\dist,\meas)$ and apply \eqref{eq:BE} and the inequality
${\rm lip}(g)\leq {\rm Lip}_a(g)$ to get
$$
|\rmD h_t f|(X)\leq\int_X|\nabla h_t f|\di\meas\leq e^{-Kt}\int_X|\nabla f|\di\meas.
$$
Letting $t\downarrow 0$ provides the inequality $\leq$ in (a). In order to prove the converse inequality we have to bound
from below the number $L$ in \eqref{eq:liminf}
along all sequences $(f_n)\subset\Lipb(X)$ convergent to $f$ in $L^1(X,\meas)$. It is not restrictive to assume that the
$\liminf$ is a finite limit and also, since $f$ is bounded, that $f_n$ are uniformly bounded. The finiteness of
$\int_X |\nabla f_n|\di\meas$ gives immediately $f_n\in H^{1,2}(X,\dist,\meas)$. In addition, for all $t>0$ it is easily seen that
$h_t f_n$ weakly converge to $h_t f$ in $H^{1,2}(X,\dist,\meas)$, hence the convexity of 
$$
g\mapsto \int_X|\nabla g|\di\meas\qquad g\in H^{1,2}(X,\dist,\meas)
$$
and Mazur's lemma give
$$
L\geq e^{Kt}\liminf_{n\to\infty}\int_X |\nabla (h_t f_n)|\di\meas\geq e^{Kt}\int_X|\nabla h_t f|\di\meas.
$$
We can use the lower semicontinuity of the total variation to get the inequality $\geq$ in (a).

The proof of (b) in the case of bounded functions uses \eqref{eq:BE} as in the proof of (a) and it is omitted. The general
case can be recovered by a truncation argument. 

The proof of (c) is an immediate consequence of \eqref{eq:Ledoux} and the definition of $BV$.
\end{proof}

The following theorem provides the stability of the $BV$ property under mGH-convergence. It will be generalized in
Theorem~\ref{thm:p-mosco}, but we prefer to give a direct proof in the $BV$ case, while the proof of
Theorem~\ref{thm:p-mosco} will focus more on the Sobolev case.

\begin{theorem}[Stability of the $BV$ property under mGH convergence]\label{thm:main} 
Let $(X,\dist,\meas_i)$ be $RCD(K,\infty)$ spaces satisfying \eqref{eq:uniform_growth}.
If $f_i\in BV(X,\dist,\meas_i)$ $L^1$-strongly converge to $f$ with $\sup_i|\rmD f_i|_i(X)<\infty$, then $f\in BV(X,\dist,\meas)$ and
\begin{equation}\label{eq:basic_lsc}
|\rmD f|(X)\leq\liminf_{i\to\infty}|\rmD f_i|_i(X).
\end{equation}  
\end{theorem}
\begin{proof} In the proof it is not restrictive to assume that the functions $f_i$ are uniformly bounded.
Indeed, since the truncated functions $f_i^N:=N\wedge f_i\lor -N$ $L^1$-converge to $f^N:=N\wedge f\lor -N$,
if we knew that $f_N\in BV(X,\dist,\meas)$, with 
$$
|\rmD f^N|(X)\leq\liminf_{i\to\infty}|\rmD f_i^N|_i(X),
$$
then we could apply \eqref{eq:monotot} to $f_i^N$ and use the lower semicontinuity of the total variation to obtain \eqref{eq:basic_lsc}.

After this reduction to uniformly bounded sequences, let us fix $t>0$ and 
consider the functions $h^i_t f_i$, which are uniformly bounded, uniformly Lipschitz (thanks to \eqref{eq:regu1}), in
$H^{1,2}(X,\dist,\meas_i)$ and converge to $h_t f\in H^{1,2}(X,\dist,\meas)$. If we were able to
prove 
\begin{equation}\label{eq:basiclsc1}
|\rmD h_tf|(X)\leq\liminf_{i\to\infty}|\rmD h^i_t f_i|_i(X)
\end{equation}
then we could use \eqref{eq:BVcontraction} to obtain $$|\rmD h_tf|(X)\leq e^{-Kt}\liminf_{i\to\infty}|\rmD f_i|_i(X)$$ and we could
eventually use once more the lower semicontinuity of the total variation to conclude.

Thanks to these preliminary remarks, in the proof of the proposition it is not restrictive to assume that $f_i$ are equi-bounded 
and equi-Lipschitz, with $f_i\in H^{1,2}(X,\dist,\meas_i)$, $f\in H^{1,2}(X,\dist,\meas)$. Assuming also with no loss of generality that 
the $\liminf$ in \eqref{eq:basic_lsc} is a finite limit, we have that $f_i$ are equi-bounded in $H^{1,2}$, 
so that they converge weakly to $f$ in $H^{1,2}$. Hence, thanks to the representation
\eqref{eq:regu2} of the total variation on Lipschitz functions, we need to prove that 
\begin{equation}\label{eq:basiclsc2}
\int_X|\nabla f|\di\meas\leq\liminf_{i\to\infty}\int_X|\nabla f_i|_i\di\meas_i.
\end{equation}
This is a consequence of Lemma~\ref{lem:old_sci} with $g\equiv 1$.
\end{proof}

\section{Compactness in $H^{1,p}$ and in $BV$} \label{sec:7}

In this section, building upon the basic compactness result in $H^{1,2}$ of \cite{GigliMondinoSavare13}, we provide new
compactness results. In order to state them in global form (i.e. passing from $L^p_\loc$-strong to $L^p$-strong
convergence) and in order to reach 
exponents $p$ smaller than 2, suitable uniform isoperimetric estimates along the sequence of spaces will be needed.

\begin{definition}[Isoperimetric profile] Assume $\meas(X)=1$. We say that $\omega:(0,\infty)\to (0,1/2]$ 
is an isoperimetric profile for $(X,\dist,\meas)$ if for all $\epsilon>0$ one has the implication
\begin{equation}\label{eq:weakisop}
\meas(A)\leq\omega(\epsilon)\qquad\Longrightarrow\qquad \meas(A)\leq\epsilon|\rmD \chi_A|(X)
\end{equation}
for any Borel set $A\subset X$.
\end{definition}

A stronger formulation is 
$$
\meas(A)\leq \Phi(|\rmD \chi_A|(X))\quad\text{whenever $\meas(A)\leq 1/2$}
$$
for some $\Phi: [0,\infty]\to [0,1]$ nondecreasing with $\Phi(0)=0$ and $\Phi(u)=o(u)$ as $u\downarrow 0$, but the formulation
\eqref{eq:weakisop}, which involves only the control of sets with sufficiently small measure, is more adapted to our needs.

If $(X,\dist,\meas)$ has $\omega$ as isoperimetric profile, one has the following property: 
for any $\epsilon>0$ and any $t\in\setR$ such that $\meas(\{f>t\})\leq\omega(\epsilon)$, one has
\begin{equation}\label{eq:median_decay}
\int_{\{f\geq t\}}(f-t)^p\di\meas\leq p^p\epsilon^p\int_X{\rm lip}^p(f)\di\meas.
\end{equation}
In order to prove \eqref{eq:median_decay} it is sufficent to apply \eqref{eq:coarea} to get
$$
\int_{\{g\geq 0\}}g\di\meas\leq\epsilon\int_X{\rm lip}(g)\di\meas\qquad\text{whenever $\meas(\{g>0\})\leq\omega(\epsilon)$.}
$$
Eventually, by applying this to $g=[(f-t)^+]^p$, with the H\"older inequality we conclude. By the definition of $\Ch_p$ we
also get
\begin{equation}\label{eq:median_decay_Chp}
\int_{\{f\geq t\}}(f-t)^p\di\meas\leq p^{p+1}\epsilon^p\Ch_p(f)\quad\text{$\forall f\in H^{1,p}(X,\dist,\meas)$ with
$\meas(\{f>t\})\leq\omega(\epsilon)$.}
\end{equation}
%
The following theorem provides classes of spaces for which the existence of an isoperimetric profile is
known.  Notice that $RCD(K,N)$ spaces with $K>0$ and $N<\infty$ have always finite diameter.

\begin{theorem}[Isoperimetric profiles]\label{thm:isopro}
The class of spaces $(X,\dist,\meas)$ with $\meas(X)=1$ having an isoperimetric profile includes:
\begin{itemize}
\item[(a)] $RCD(K,\infty)$ spaces with $K>0$;
\item[(b)] $RCD(K,\infty)$ spaces with finite diameter.
\end{itemize}
\end{theorem}
\begin{proof} Statement (a) follows from Bobkov's inequality that, when particularized to characteristic functions, gives
$\sqrt{K}{\cal I}(\meas(A))\leq |\rmD\chi_A|(X)$, where ${\cal I}$ is the Gaussian isoperimetric function.
The proof given in \cite[Theorem~8.5.3]{BakryGentilLedoux} can be adapted without great difficulties to the context of $RCD(K,\infty)$ metric measure spaces (notice that the setting of Markov triples of 
\cite{BakryGentilLedoux}, with a $\Gamma$-invariant algebra of functions, does not seem to apply to $RCD(K,\infty)$ spaces),
see \cite{AmbrosioMondino16} for a proof.

Statement (b) is a direct consequence of Proposition~\ref{prop:quasiiso} and of the definition of $BV$ which,
choosing $f=\chi_A$, grant the inequality
$$
\meas(A)\leq \epsilon\bigl(\meas(A)+|\rmD \chi_A|(X)\bigr)
$$
as soon as $M(\epsilon,D,K)\meas(A)\leq 1$.
\end{proof}

\begin{remark}[Sharp isoperimetric profiles]{\rm
See also \cite{CavallettiMondino15a} for comparison results and for a description of the sharp isoperimetric
profile in the case when $N<\infty$, in the much more general class of $CD(K,N)$ spaces
(assuming finiteness of the diameter when $K\leq 0$).
}
\end{remark}

The following compactness theorem is one of the main results of \cite{GigliMondinoSavare13}, see
Theorem~6.3 therein, we just adapted a bit the statement to our needs, adding also a compactness
in $L^2_\loc$ independent of the equi-tightness condition \eqref{eq:tightgms}. We say that a sequence
$(f_i)$ $L^2_\loc$-strongly converges to $f$ is $f_i\phi$ $L^2$-strongly converges to $f\phi$ for all
$\phi\in\Cbs(X)$.

\begin{theorem} \label{thm:gms13} Assume that $(X,\dist,\meas_i)$ are $RCD(K,\infty)$ spaces and
$f_i\in H^{1,2}(X,\dist,\meas_i)$ satisfy
\begin{equation}\label{eq:W12bound}
\sup_i\int_X|f_i|^2\di\meas_i+\Ch^i(f_i)<\infty 
\end{equation}
and (for some and thus all $\bar x\in X$)
\begin{equation}\label{eq:tightgms}
\lim_{R\to\infty}\limsup_{i\to\infty}\int_{X\setminus B_R(\bar x)}|f_i|^2\di\meas_i=0.
\end{equation}
Then $(f_i)$ has a $L^2$-strongly convergent subsequence to $f\in H^{1,2}(X,\dist,\meas)$.
In  general, if only \eqref{eq:W12bound} holds, $(f_i)$ has a subsequence $L^2_\loc$-strongly convergent to $f\in H^{1,2}(X,\dist,\meas)$.
\end{theorem}
\begin{proof} The first part, as we said, is \cite[Theorem~6.3]{GigliMondinoSavare13}. For the second part, having fixed
$\bar x\in X$, it is sufficient to apply the first part to the sequences $f_i\chi_R$, where $\chi_R\in\Lip(X,[0,1])$ with 
$\chi_R\equiv 1$ on $B_R(\bar x)$ and $\chi_R\equiv 0$ on $X\setminus B_{R+1}(\bar x)$,
and then to apply a standard diagonal argument.
\end{proof}

Under suitable finiteness assumptions, coupled with the existence of a common isoperimetric profile, we
can extend this result to $L^{p_i}$ compactness, assuming Sobolev or $BV$ bounds, as follows.
 
\begin{proposition}\label{prop:compactness} Assume that $(X,\dist,\meas_i)$, $(X,\dist,\meas)$ are $RCD(K,\infty)$ spaces
satisfying $\meas_i(X)=1$, $\meas(X)=1$ and with a common isoperimetric profile.

Assuming that $p_i>1$ converge to $p$ in $[1, \infty)$ and that $f_i\in H^{1,p_i}(X,\dist,\meas_i)$ satisfy
$$
\sup_i\int_X|f_i|^{p_i}\di\meas_i+\Ch_{p_i}^i(f_i)<\infty,
$$
the family $(f_i)$ has a $L^{p_{i(j)}}$-strongly convergent subsequence $(f_{i(j)})$. Analogously, if $p_i=1$ and 
$$
\sup_i\int_X|f_i|\di\meas_i+|\rmD f_i|_i(X)<\infty,
$$
then the family $(f_i)$ has a $L^1$-strongly convergent subsequence $(f_{i(j)})$.
\end{proposition}
\begin{proof} By $L^{p_i}$-weak compactness we can assume that the weak limit $f\in L^p(X,\meas)$ exists.

The case $p_i=2$ for infinitely many $i$ is already covered by Theorem~\ref{thm:gms13}, indeed
the condition \eqref{eq:tightgms} is automatically satisfied under the isoperimetric assumption, splitting
$$
\int_{X\setminus B_R(\bar x)}|f_i|^2\di\meas_i\leq
\int_{\{|f_i|\geq M\}}|f_i|^2\di\meas_i+M^2\meas_i(X\setminus B_R(\bar x))
$$
and using \eqref{eq:median_decay_Chp} with $p=2$, letting first $R\to\infty$ and then $M\uparrow\infty$.

Hence, in the sequel we need
only to consider the cases $p_i>2$ for $i$ large enough and $p_i<2$ for $i$ large enough. 
  
In the case when $p_i>2$ for $i$ large enough the proof is simpler, since for any $\delta>0$ we can write $f_i=g_i+h_i$ with
$\|h_i\|_{L^{p_i}(X,\meas_i)}<\delta$, $\|g_i\|_{L^\infty(X,\meas_i)}$ equibounded and $\sup_i\Ch^i_{p_i}(g_i)<\infty$. Since
$2\Ch_2^i(g_i)\leq \bigl( p_i\Ch_{p_i}^i(g_i)\bigr)^{2/p_i}$, it follows that $\Ch_2^i(g_i)$ is bounded as well, hence by what we already
proved in the case $p=2$ we can find a subsequence $g_{i(j)}$ $L^2$-strongly convergent and then
(since $(g_i)$ are equibounded) $L^{p_i}$-strongly convergent.  The decomposition $f_i=g_i+h_i$ can be achieved using
\eqref{eq:median_decay_Chp} with $p=p_i$, which gives 
$$
\lim_{M\to\infty}\sup_i\int_{\{|f_i|>M\}}(|f_i|-M)^{p_i}\di\meas_i=0.
$$
This is due to the fact that Markov's inequality and the uniform $L^1$ bound on $f_i$ give
$$
\lim_{M\to\infty}\sup_i\meas_i(\{|f_i|>M\})=0.
$$
Hence, we can first choose $\epsilon>0$ sufficiently small, in such a way that $\sup_i p_i^{p_i+1} \epsilon^{p_i}\Ch^i_{p_i}(f_i)<\delta$ 
and then $M$ in such a way that $\sup_i\meas_i(\{|f_i|\geq M\})\leq\omega(\epsilon)$, setting $g_i=(f_i\lor -M)\land M$. 

In the case $p_i<2$ for $i$ large enough the decomposition $f_i=g_i+h_i$ can still be achieved using \eqref{eq:median_decay_Chp}
(with $\epsilon\sup_i |\rmD f_i|(X)<\delta$ in the case $p_i=1$).
Since $p_i<2$, this time we need one more regularization step to achieve the compactness of $g_i$. More precisely, we write
$g_i=(g_i-h^i_t g_i)+h^i_tg_i$; since $h^i_t g_i$ are uniformly Lipschitz we obtain that $\sup_i\Ch_2(h^i_tg_i)$
is uniformly bounded, hence we can extract a $L^2$-strongly convergent (and also $L^{p_i}$-strongly convergent) 
subsequence. It remains to prove that
\begin{equation}\label{eq:got_stuck}
\lim_{t\downarrow 0}\limsup_{i\to\infty}\int_X|g_i-h^i_t g_i|^{p_i}\di\meas_i=0.
\end{equation}
This is an immediate consequence of \eqref{eq:Ledoux1} and the uniform boundedness of $(g_i)$.
\end{proof}

\section{Mosco convergence of $p$-Cheeger energies}\label{sec:8}

The definition of Mosco convergence can be immediately adapted to the case when the exponent $p$ is different from $2$
and even $i$-dependent. Adopting the convention $\Ch_1(f)=|\rmD f|(X)$ to include also the case $p=1$, if $p_i\in [1,\infty)$
converge to $p\in [1,\infty)$ we say that the 
$p_i$-Cheeger energies $\Ch^i_{p_i}$ relative to $(X,\dist,\meas_i)$ Mosco converge to $\Ch_p$, the
$p$-Cheeger energy relative to $(X,\dist,\meas)$, if:

\begin{itemize}
\item[(a)] (\emph{Weak-$\liminf$}). For every $f_i\in L^{p_i}(X,\meas_i)$ $L^{p_i}$-weakly converging to $f\in L^p(X,\meas)$, one has
\[ \Ch_p(f)\le \liminf_{i\to\infty} \Ch^i_{p_i}(f_i).\]
\item[(b)] (\emph{Strong-$\limsup$}). For every $f \in L^p(X,\meas)$ there exist $f_i\in L^{p_i}(X,\meas_i)$ $L^{p_i}$-strongly converging to $f$ with
\begin{equation}\label{eq:optimal_Chee_pi} \Ch_p(f)=\lim_{i\to \infty} \Ch^i_{p_i}(f_i).\end{equation}
\end{itemize}

We speak instead of $\Gamma$-convergence if the same notions of convergence occurs in (a) and (b), namely the liminf inequality
is only required along $L^{p_i}$-strongly convergent sequences. Obviously Mosco convergence implies $\Gamma$-convergence and we have
provided in Proposition~\ref{prop:compactness} a compactness result that allows to improve, under the assumptions on $(X,\dist,\meas_i)$ 
stated in the proposition, $\Gamma$ to Mosco convergence.

\begin{theorem} \label{thm:p-mosco}
Let $(X,\dist,\meas_i)$ be $RCD(K,\infty)$ spaces satisfying \eqref{eq:uniform_growth} and
let $(p_i)\subset [1,\infty)$ be convergent to $p\in [1,\infty)$. Then $\Ch^i_{p_i}$ $\Gamma$-converge to
$\Ch_p$. Under the assumption of Proposition~\ref{prop:compactness} one has Mosco convergence.
\end{theorem}

\begin{proof} {\it $\liminf$ inequality, $p>1$.} Possibly replacing $f_i$ by their $L^{p_i}$ approximations involved in the 
definition of $\Ch_{p_i}$, we need only to prove the weaker inequality
\begin{equation}\label{eq:basiclsc16}
p\Ch_p(f)\leq\liminf_{i\to\infty}\int_X{\rm Lip}_a^{p_i}(f_i)\di\meas_i.
\end{equation}

Assume first that $f_i$ are uniformly bounded in $H^{1,2}$ and equi-Lipschitz. 
Then, Lemma~\ref{lem:old_sci} and the inequality $|\nabla f|_i\leq {\rm lip}(f)$ give
$$
\int_Xg|\nabla f|\di\meas\leq\liminf_{i\to\infty}\int_Xg|\nabla f_i|_i\di\meas_i
\leq\liminf_{i\to\infty}\int_Xg{\rm lip}(f_i)\di\meas_i\quad
$$
for any $g$ lower semicontinuous and nonnegative. This, in combination with the elementary duality identity
\begin{equation}\label{eq:duality}
\frac 1p\int_X |\nabla f|^p \di\meas=\sup\left\{\int_X g|\nabla f|\di\meas-\frac {1}{q}\int_X g^q\di\meas:\
g\in \Cbs(X),\,\,g\geq 0\right\}
\end{equation}
with $q$ dual exponent of $p$ (applied also to the spaces $(X,\dist,\meas_i)$ with $p=p_i$) provides the inequality
\begin{equation}\label{eq:basiclsc166}
\int_X|\nabla f|^p\di\meas\leq\liminf_{i\to\infty}\int_X{\rm Lip}_a^{p_i}(f_i)\di\meas_i.
\end{equation}

In order to remove the additional assumptions on $f_i$ we now consider the intermediate case when $f_i$ are uniformly
bounded in $L^\infty$ and in $L^2$. Let us fix $t>0$ and consider the functions $h^i_t f_i$, which are uniformly bounded, 
uniformly Lipschitz (thanks to \eqref{eq:regu1}), in $H^{1,2}(X,\dist,\meas_i)$ and weakly converge in $H^{1,2}$ to $h_t f\in H^{1,2}(X,\dist,\meas)$ 
by Theorem~\ref{thm:gms13_flow}. 
Then we can use \eqref{eq:regu1}, \eqref{eq:BE} and \eqref{eq:basiclsc166} with $h^i_t f_i$ to get
$$
e^{Kpt}\int_X{\rm Lip}_a^p(h_{2t} f)\di\meas\leq
\int_X|\nabla h_t f|^p\di\meas\leq e^{-Kpt}\liminf_{i\to\infty}\int_X{\rm Lip}_a^{p_i}(f_i)\di\meas_i.
$$ 
Letting $t\downarrow 0$ then provides \eqref{eq:basiclsc16}.

Eventually we consider the general case $f_i$; possibly splitting in positive and negative parts, we assume $f_i\geq 0$. 
We consider the truncation $1$-Lipschitz functions (notice that the quadratic regularization near the origin is necessary
in the case $p\geq 2$, to get $L^2$ integrability)
$$
\phi_N(t):=
\begin{cases}
\frac N 2 z^2 &\text{if $0\leq z\leq\frac 1N$};\\
-\frac 1{2N}+z&\text{if $\frac 1N\leq z\leq N$};\\
-\frac 1{2N}+N&\text{if $N\leq z$}
\end{cases}
$$
and $f_i^N:=\phi_N\circ f_i$. Since $f_i^N$ $L^{p_i}$-strongly converge to $f^N:=\phi_N\circ f$, hence
$$
\Ch_p(f^N)\leq\liminf_{i\to\infty}\Ch^i_{p_i}(f_i^N)\leq\liminf_{i\to\infty}\Ch^i_{p_i}(f_i).
$$
By letting $N\to\infty$ we conclude.

\noindent
{\it $\liminf$ inequality, $p=1$.} The proof is analogous, in the case when the $f_i$ are uniformly bounded 
it is sufficient to prove \eqref{eq:basiclsc16} for the regularized functions $h_t^i f_i$, $h_t f$, without 
using the duality formula \eqref{eq:duality}. Eventually the uniform boundedness assumption on $f_i$ can be removed as
in the case $p>1$, with the simpler truncations $\phi_N(z)=\min\{N,x\}$.

\noindent
{\it $\limsup$ inequality.} For $p>1$, let us consider $f\in H^{1,p}(X,\dist,\meas)$ 
and $f^N\in\Lip_\bs(X)$ with ${\rm Lip}_a(f^N)\to |\nabla f|$ in $L^p(X,\meas)$. For any $N$ one has, by the upper semicontinuity of the
asymptotic Lipschitz constant 
$$
\limsup_{i \to \infty}p_i\Ch^i_{p_i}(f^N)\leq\limsup_{i\to\infty}\int_X{\rm Lip}_a^{p_i}(f^N) \di\meas_i\le 
\int_X{\rm Lip}^p_a(f^N)\di\meas.
$$
Since $f^N$ $L^{p_i}$ converge to $f^N$, by a diagonal argument, we can then define $f_i=f^{N(i)}$ with $N(i)\to\infty$ as $i\to\infty$ in such a way
that $f_i$ $L^{p_i}$ converge to $f$ and
$\limsup_i\Ch_{p_i}^i(f_i)\leq\Ch_p(f)$. For $p=1$ the proof is similar and uses Proposition~\ref{prop:better_var}.
\end{proof}

\section{$p$-spectral gap}\label{sec:9}

Throughout this section we assume that $\meas(X)=1$ when a single space is considered and, when a sequence
is considered, also $\meas_i(X)=1$. For any $p \in [1, \infty)$ and any $f \in L^p(X,\meas)$ we put
\begin{equation}\label{eq:def cp}
c_p(f):= \left(\inf_{a \in\setR}\int_X|f-a|^p\di\meas\right)^{1/p}.
\end{equation}

We also recall that for any $f \in L^1(X,\meas)$ there exists a \textit{median of $f$}, i.e.
a real number $m$ such that
$$
\meas\left(\{f > m\right\})  \le\frac 12
\qquad\text{and}\qquad
\meas\left( \{f < m\}\right) \le \frac 12.
$$

In the following remark we recall a few well-known facts about the minimization problem \eqref{eq:def cp}
(see also \cite[Lemma 2.2]{WuWangZheng}, \cite{Chavel}).

\begin{remark}\label{rem:pmedians}
For $p\in (1,\infty)$, thanks to the strict convexity of $z\mapsto |z|^p$ there is a unique minimizer $a$ in \eqref{eq:def cp},
and it is characterized by
$$\int_X|f-a |^{p-2}(f-a )\di\meas=0.$$
It is also well-known that, when $p=1$, medians are minimizer in \eqref{eq:def cp}, the converse seems to be less well-known, so let us
provide a simple proof. Assume that $a$ is a minimizer and assume by contradiction that
$\meas(\{f>a\})>1/2$ (if $\meas(\{f<a\})>1/2$ the argument is similar). We can then
find $\delta>0$ such that $\meas(\{f>a+\delta\})>1/2$ and a simple computation gives
\begin{eqnarray*}
\int_X|f-(a+\delta)|\di\meas-\int_X|f-a|\di\meas&=&
\delta\bigl(\meas(\{f<a+\delta\})-\meas(\{f\geq a+\delta\})\bigr)\\&-&2\int_{\{a<f<a+\delta\}}(f-a)\di\meas
<0,
\end{eqnarray*}
contradicting the minimality of $a$.
\end{remark}

In particular, for any $p \in [1, \infty)$ there exists a minimizer of \eqref{eq:def cp}, and it will be denoted by $m_p(f)$; by
convention, it will be any median of $f$ when $p=1$. Analogously, when we say that $m_{p_i}(f_i)$ converge to $m_p(f)$
we understand this convergence in the set-theoretic sense when $p=1$ (i.e. limit points of $m_{p_i}(f_i)$ are medians).

\begin{lemma}\label{lem:conv const}
Let $p_i$ converge to $p$ in $[1, \infty)$ and let $f_i \in L^{p_i}(X,\meas_i)$ be an $L^{p_i}$-strongly 
convergent sequence to $f \in L^p(X,\meas)$. Then 
$$
\lim_{i\to\infty}m_{p_i}(f_i)=m_p(f)\qquad\text{and}\qquad\lim_{i \to \infty}c_{p_i}(f_i)=c_p(f).
$$
\end{lemma}
\begin{proof}
Since 
$$
\limsup_{i \to \infty}c_{p_i}(f_i) \le \lim_{i \to \infty}\left(\int_X|f_i -b|^{p_i}\di\meas_i\right)^{1/p_i}
=\left(\int_X|f-b|^p\di\meas\right)^{1/p}\quad\forall b\in\setR,
$$
taking the infimum with respect to $b$ gives the upper semicontinuity of $c_{p_i}(f_i)$.

On the other hand, since it is easily seen that $|m_{p_i}(f_i)|\leq 2\|f_i\|_{L^{p_i}(X,\meas_i)}$, the family $m_{p_i}(f_i)$ has 
limit points as $i\to\infty$, and if $m_{p_{i(k)}}(f_{i(k)})\to a$ as $k\to\infty$ one has
\begin{eqnarray}\label{eq:e_partial}
\liminf_{k \to \infty}c_{p_{i(k)}}(f_{i(k)})&=&\liminf_{k \to \infty}\left(\int_X|f_i-m_{p_{i(k)}}(f_{i(k)})|^{p_i}\di\meas_i\right)^{1/p_i}\nonumber
\\&=&\left(\int_X|f-a|^p\di\meas\right)^{1/p} \ge c_p(f).
\end{eqnarray}
If we apply this to limit points of subsequences $i(k)$ on which the $\liminf_k c_{p_{i(k)}}(f_{i(k)})$ is achieved, this gives that $c_{p_i}(f_i)\to c_p(f)$.
In addition, the inequality \eqref{eq:e_partial} gives that any limit point of $m_{p_i}(f_i)$ is a minimizer.
\end{proof}

Now, for $p\in [1,\infty)$ let
\begin{equation}\label{eq:def p-lap}
\lambda_{1, p}(X,\dist,\meas):=\inf_f \frac{1}{c_p^p(f)}\int_X{\rm Lip}_a^p(f)\di\meas,
\end{equation}
where the infimum runs among all nonconstant Lipschitz functions $f$ on $X$.
By the very definition of $\Ch_p$, the infimum above does not change if we minimize $p\Ch_p(f)/c_p^p(f)$
in the class of nonconstant functions $f\in H^{1,p}(X,\dist,\meas)$. Furthermore, whenever a minimizer exists,
we may normalize it in such a way that $c_p(f)=\|f\|_{L^p(X,\meas)}=1$ (i.e. the infimum in \eqref{eq:def cp} is attained at $a=m_p(f)=0$).

For $p \in (1, \infty)$, Remark~\ref{rem:pmedians} and the definition of $\Ch_p$ gives other characterizations of $\lambda_{1, p}(X)$:
\begin{align}\label{eq:ot ch p-lap}
\lambda_{1, p}(X,\dist,\meas)
&=\inf \left\{ \int_X{\rm Lip}_a^p(f)\di\meas:\ f \in \mathrm{Lip}(X,\dist), \int_X|f|^p\di\meas=1,\,\, \int_X|f|^{p-2}f\di\meas=0\right\} \nonumber \\
&=\inf \left\{ \int_X{\rm lip}^p(f)\di\meas:\ f \in \mathrm{Lip}(X,\dist), \int_X|f|^p\di\meas=1,\,\, \int_X|f|^{p-2}f\di\meas=0\right\} \nonumber \\
&=\inf \left\{ p\Ch_p(f):\ f \in H^{1, p}(X,\dist,\meas), \int_X|f|^p\di\meas=1, \int_X|f|^{p-2}f\di\meas=0\right\}.
\end{align}

\begin{remark} If $\meas(X)=1$,
let us define \textit{the Cheeger constant $h(X,\dist,\meas)$ of $(X,\dist,\meas)$} by 
$$
h(X,\dist,\meas):=\inf_A \frac{M^-(A)}{\meas(A)},
$$
where the infimum runs among all Borel subsets $A$ of $X$ with $0<\meas(A) \le 1/2$, and $M^-(A)$ is the
lower Minkowski content of $A$, namely (here $I_r(A)$ is the open $r$-neighbourhood of $A$)
$$
M^-(A):=\liminf_{r \to 0^+}\frac{\meas\left(I_r(A)\right)-\meas(A)}{r}.
$$
Then, in \cite{AmbrosioDiMarinoGigli} it has been proved that
$$
h(X,\dist,\meas)=\inf_A \frac{|\rmD\chi_A|(X)}{\meas(A)},
$$
where as before  the infimum runs among all Borel subsets $A$ of $X$ with $0<\meas(A) \le\meas(X)/2$ (the same
result holds if we use the upper Minkowski content in the definition of $h$). On the other hand,
by applying Lemma~\ref{lem:conv const} with $\meas_i=\meas$, from Proposition~\ref{prop:better_var} we get
\begin{equation}\label{eq:altrolambda}
\lambda_{1,1}(X,\dist,\meas)=\inf\left\{ \frac{|\rmD f|(X)}{c_1(f)}:\ f \in BV(X,\dist,\meas),\,\, f \not \equiv \mathrm{constant}\right\}.
\end{equation}
Eventually, since $c_1(\chi_A)=\meas(A)$ for $\meas(A)\leq 1/2$,
the coarea formula for $BV$ maps shows that the Cheeger constant $h$ coincides also with the quantities in \eqref{eq:altrolambda}.
\end{remark}

In the following theorem we prove a generalized continuity property \eqref{equ:gen cont} of the first eigenvalue,
allowing also the exponents $p_i\to p\in [1,\infty)$ to depend on $i$. As the proof shows, this property holds even in the 
extreme case when $\mathrm{diam}\,\mathrm{supp}(\meas)=0$, with the convention 
$$(\lambda_{1, p}(X,\dist, \meas))^{1/p}:=\infty\qquad\text{if $\mathrm{diam}\,\mathrm{supp}(\meas)=0$.}$$
Note that  \eqref{equ:gen cont} in the case when $\mathrm{diam}\,\mathrm{supp}(\meas)=0$ will be used in the proof of 
Corollary~\ref{cor;shperical suspension}.

\begin{theorem}\label{thm:p-spec}
Assume that $(X,\dist,\meas_i)$, $(X,\dist,\meas)$ are $RCD(K,\infty)$ spaces
satisfying $\meas_i(X)=1$, $\meas(X)=1$ with a common isoperimetric profile (for
instance either $K>0$ or uniformly bounded diameters of $\supp\meas_i$).
If $p_i$ converge to $p$ in $[1, \infty)$, then
\begin{equation} \label{equ:gen cont}
\lim_{i \to \infty}\lambda_{1, p_i}(X,\dist,\meas_i)=\lambda_{1, p}(X,\dist,\meas).
\end{equation}
In particular the Cheeger constants are continuous with respect to the measured Gromov-Hausdorff convergence.
\end{theorem}
\begin{proof}
For any $f \in H^{1,p}(X,\dist,\meas)$ with $c_p(f)=\|f\|_p=1$, by Theorem~\ref{thm:p-mosco}, there exists a sequence 
$f_i \in H^{1,p_i}(X,\dist,\meas_i)$ $L^{p_i}$-strongly converging to $f$ with $\limsup_i\Ch^i_{p_i}(f_i)\leq\Ch_p(f)$.
Applying Lemma~\ref{lem:conv const} yields
$$
\limsup_{i \to \infty}\lambda_{1, p_i}(X,\dist,\meas_i) \le \limsup_{i \to \infty}\frac{p_i\Ch^i_{p_i}(f_i)}{\left(c_{p_i}(f_i)\right)^{p_i}}
\le \Ch_p(f).
$$
Taking the infimum with respect to $f$ gives the upper semicontinuity of $\lambda_{1, p_i}(X,\dist,\meas_i)$.

In order to prove the lower semicontinuity, we can assume with no loss of generality that $\lambda_{1, p_i}(X,\dist,\meas_i)$
is a bounded convergent sequence. For any $i\geq 1$ take $f_i \in H^{1,p_i}(X,\dist,\meas_i)$ with 
$$
\left| \lambda_{1, p_i}(X,\dist,\meas_i)-p_i\Ch^i_{p_i}(f_i) \right|<\frac 1 i
\quad\text{and}\quad
c_{p_i}(f_i)=\int_X|f_i|^{p_i}\di\meas_i=1.
$$
By Proposition~\ref{prop:compactness}, without loss of generality we can assume that the 
$L^{p_i}$-strong limit $f \in L^p(X,\meas)$ of $f_i$ exists. 
Thus, Theorem~\ref{thm:main} gives $\Ch_p(f)\leq\liminf_i\Ch^i_{p_i}(f_i)$.
As a consequence, since Lemma~\ref{lem:conv const} gives $c_p(f)=\|f\|_{L^p(X,\meas)}=1$, we have
$$
\liminf_{i \to \infty}\lambda_{1, p_i}(X,\dist,\meas_i)=\liminf_{i \to \infty}p_i\Ch_{p_i}^i(f_i) \ge p\Ch_p(f)\geq\lambda_{1,p}(X,\dist,\meas).
$$
\end{proof}

For $p\in (1,\infty)$ and $\Omega\subset X$ Borel, let us denote 
$$
\Lambda_p(\Omega,\dist,\meas):=\left\{f \in H^{1, p}(X,\dist,\meas):\ \int_\Omega|f|^p\di\meas=1,\,\, \text{$f=0$ $\meas$-a.e. on $X \setminus \Omega$}\right\}.
$$
Accordingly, we define $\lambda_{1, p}^D(\Omega,\dist,\meas)$ as the infimum of the $p$-energy with Dirichlet conditions
\begin{equation}\label{eq:def dir p-lap}
\lambda_{1, p}^D(\Omega,\dist,\meas):=\inf \left\{ p\Ch_p(f):\ f\in\Lambda_p(\Omega,\dist,\meas)\right\}.
\end{equation}

\begin{lemma}\label{lem:min max p-lap} Let $p\in (1,\infty)$.
\begin{enumerate}
\item[(1)] For any Borel subsets $\Omega_1, \,\Omega_2$ of $X$ with $\meas(\Omega_1 \cap \Omega_2)=0$, we have 
\begin{equation}\label{eq:minmax}
\lambda_{1, p}(X,\dist,\meas)\le \max \left\{\lambda_{1, p}^D(\Omega_1,\dist,\meas), \lambda_{1, p}^D(\Omega_2,\dist,\meas) \right\}.
\end{equation}
\item[(2)] If $p \in [2, \infty)$ and $f \in H^{1, p}(X,\dist,\meas)$ is a minimizer of the right hand side of \eqref{eq:def p-lap} with $m_p(f)=0$, then 
\begin{equation}\label{eq:p-lap eq}
\int_X\langle \nabla f, \nabla g \rangle |\nabla f|^{p-2}\di\meas =\lambda_{1, p}(X, \dist, \meas)\int_X|f|^{p-2}fg\di\meas
\end{equation}
for any $g \in H^{1, p}(X, \dist, \meas)$. In particular, choosing $g=f^\pm$ gives
\begin{equation}\label{eq:glob loc}
\lambda_{1, p}(X,\dist,\meas)=p\Ch_p(f^\pm)\left(\int_X|f^\pm|^p\di\meas\right)^{-1}.
\end{equation}
\end{enumerate}
\end{lemma}
\begin{proof} We first prove \eqref{eq:minmax}.
Take $f_i \in H^{1, p}(X,\dist,\meas)$ with $\int_{\Omega_i}|f_i|^p\di\meas=1$ and $f_i=0$ $\meas$-a.e. on $X\setminus\Omega_i$.
Then, choosing thanks to a continuity argument $\alpha\in\setR$ such that $\int_X|f_1+\alpha f_2|^{p-2}(f_1+\alpha f_2)\di\meas=0$, we get
\begin{align}\label{eq:est}
(1+|\alpha|^p)\lambda_{1,p}(X,\dist,\meas)&=
\lambda_{1, p}(X,\dist,\meas) \biggl(\int_{\Omega_1}|f_1|^p\di\meas+\int_{\Omega_2}|\alpha f_2|^p\di\meas\biggr)\nonumber \\&= 
\lambda_{1, p}(X,\dist,\meas)\int_X|f_1+\alpha f_2|^p\di\meas \\
&\leq p\Ch_p(f_1+\alpha f_2) =p\Ch_p(f_1)+p|\alpha|^p\Ch_p(f_2).\nonumber
\end{align}
By taking the infimum w.r.t. $f_1$ and $f_2$ we obtain \eqref{eq:minmax}.

Next we prove \eqref{eq:p-lap eq}.
Let 
$$
F(s, t):=\int_X|f+sg-t|^{p-2}(f+sg-t)\di\meas.
$$
Then, it is easy to check that 
$$
F_s(s, t)=(p-1)\int_Xg|f+sg-t|^{p-2}\di\meas
$$
and that
$$
F_t(s, t)=(1-p)\int_X|f+sg-t|^{p-2}\di\meas,
$$
the implicit function theorem yields that $s\mapsto m_p(f+sg)$ is differentiable at $s=0$.

Now, recall that according to \cite{GigliHan}, we can represent $p\Ch_p(f)$ as $\int_X|\nabla f|^p\di\meas$, where
$|\nabla f|$ is the $2$-minimal relaxed slope (as always, in this paper). 
Then, the direct calculation of the left hand side of
$$
\frac{\di}{\di s}\left(\frac{p\Ch_p(f+sg)}{\|(f+sg)-m_p(f+sg)\|^p_{L^p(X,\meas)}}\right)\Biggl\vert_{s=0}=0
$$ 
with the differentiability of $m_p(f+sg)$ at $s=0$ proves (\ref{eq:p-lap eq}).
\end{proof}

In the following stability result we need the extra assumption
\begin{equation}\label{eq:extra_assumption}
\begin{gathered}
\limsup_{i\to\infty}\|f_i\|_{L^{p_i}(X,\meas_i)}\leq\|f\|_{L^\infty(X,\meas)}\\
\text{whenever $p_i\to\infty$, $\sup_i\|f_i\|_{L^{p_i}(X,\meas_i)}+\bigl(\int_X|\nabla f_i|^{p_i}\di\meas_i\bigr)^{1/p_i}<\infty$}\\
\text{and $f_i$ strongly $L^p$-converge to $f$ for some (and thus all) $p\in (1,\infty)$.}
\end{gathered}
\end{equation}
This is a kind of extension of Theorem~\ref{thm:p-spec} to the case $p=\infty$.
We believe that it should be possible to avoid this assumption,
possibly making an additional hypothesis on the decay rate of the common isoperimetric profile. Nevertheless, this assumption
is harmless for the applications of Theorem~\ref{thm:gros} below in Section~\ref{sec:11}. Indeed, in the setting of Section~\ref{sec:11},
as soon as $p_i>N$ the functions $f_i$ and $f$ are equibounded and equi-H\"older on $\supp\meas_i$, $\supp\meas$
respectively; denoting by $f_i,\,f$ suitable equibounded and equi-H\"older extensions of $f_i,\,f$ to the whole of $X$, the 
Hausdorff convergence of $\supp\meas_i$ to $\supp\meas$ and the weak convergence of $f_i\meas_i$ to $f\meas$ 
easily imply the uniform convergence of $f_i$ to $f$ on $\supp\meas$, so that
$$
\limsup_{i\to\infty}\|f_i\|_{L^{p_i}(X,\meas_i)}\leq\limsup_{i\to\infty}\|f_i\|_{L^{p_i}(X,\meas)}\leq
\limsup_{i\to\infty}\|f\|_{L^{p_i}(X,\meas)}\leq \|f\|_{L^\infty(X,\meas)}.
$$

\begin{theorem}\label{thm:gros}
Let $(X,\dist,\meas_i)$, $(X,\dist,\meas)$ be $RCD(K,\infty)$ metric measure spaces with $\meas_i(X)=1$, $\meas(X)=1$ 
and a common isoperimetric profile (e.g. either $K>0$ or equibounded diameters of $\supp\meas_i$). If $p_i\in [1,\infty)$ diverge to $\infty$ 
and \eqref{eq:extra_assumption} holds, one has 
\begin{equation}\label{eq:upp diam}
\lim_{i \to \infty}\left( \lambda_{1, p_i}(X,\dist,\meas_i) \right)^{1/p_i}=\frac{2}{\mathrm{diam}\,\supp(\meas)}.
\end{equation}
\end{theorem}
\begin{proof}
Let $x_1,\,x_2\in\supp\meas$; thanks to the weak convergence of $\meas_i$ to $\meas$ we can find $x_{j,i}$ convergent to $x_j$ as $i\to\infty$,
$j=1,\,2$. Let $r=\dist(x_1,x_2)$, $r_i=\dist(x_{1,i},x_{2,i})$ and let us define nonnegative Lipschitz functions $\delta_{j, i}\in\Lip(X,\dist)$ by
$$
\delta_{j, i}(x):= \max \left\{ \frac{r_i}{2}-\dist(x_{j, i}, x), 0\right\},
$$
uniformly convergent as $i\to\infty$ to
$$
\delta_j(x):= \max \left\{ \frac r2-\dist(x_j, x), 0\right\}.
$$
Then, since $\{B_{r_i/2}(x_{j, i})\}_{j=1, 2}$ are nonempty disjoint subsets of $X$, and since $\delta_{j,i}$ are
$1$-Lipschitz, for any $p \in (1, \infty)$, \eqref{eq:minmax} and the H\"older inequality give that 
\begin{align*}
\left(\lambda_{1, p_i}(X,\dist,\meas_i)\right)^{1/p_i} &\le \max_{j=1,2}\left\{ \left(\lambda^D_{1, p_i}\left( B_{r_i/2}(x_{j, i})\right)\right)^{1/p_i} \right\} \\
&\le \max_{j=1,2} \left\{ \left(\frac{1}{\meas_i\left(B_{r_i/2}(x_{j, i})\right)}\int_{B_{r_i/2}(x_{j, i})}\delta_{j, i}^{p_i}\di\meas_i\right)^{-1/p_i} \right\}\\
&\le  \max_{j=1,2} \left\{ \left(\frac{1}{\meas_i\left(B_{r_i/2}(x_{j, i})\right)}\int_{B_{r_i/2}(x_{j, i})}\delta_{j, i}^p\di\meas_i\right)^{-1/p} \right\}
\end{align*}
for all sufficiently large $i$.
Thus by letting $i \to \infty$ we have
$$
\limsup_{i \to \infty}\left( \lambda_{1, p_i}(X,\dist,\meas_i)\right)^{1/p_i}\le \max_{j=1,2} 
\left\{ \left(\frac{1}{\meas\left(B_{r/2}(x_j)\right)}\int_{B_{r/2}(x_j)}\delta_j^p\di\meas\right)^{-1/p} \right\}.
$$
Letting $p \to \infty$ yields
$$
\limsup_{i \to \infty}\left( \lambda_{1, p_i}(X,\dist,\meas_i)\right)^{1/p_i}\le \max_{j=1,2} \{\|\delta_j\|_{L^\infty(X,\meas)}^{-1}\}=\frac{2}r=\frac{2}{\dist(x_1,x_2)}.
$$
By minimizing w.r.t. $x_1$ and $x_2$ we get the $\limsup$ inequality in \eqref{eq:upp diam}.

Next we check the $\liminf$ inequality in \eqref{eq:upp diam}. We can assume with no loss of generality that the limit 
$\lim_i\left( \lambda_{1, p_i}(X,\dist,\meas_i)\right)^{1/p_i}$ exists  and is finite.
For any $i$ such that $p_i>2$ take a minimizer $f_i \in H^{1, p_i}(X,\dist,\meas_i)$ of the right hand side of \eqref{eq:ot ch p-lap}
(whose existence is granted by Proposition~\ref{prop:compactness}). 
Set $\tilde f_i:=f_i^\pm/\|f_i^\pm\|_{L^{p_i}(X,\meas_i)}$ and $\tilde f_i:=\tilde f_i^+-\tilde f_i^-$.
Since Lemma~\ref{lem:min max p-lap} yields 
$$
\lambda_{1, p_i}(X,\dist,\meas_i)=p_i\Ch^i_{p_i}(\tilde f_i^\pm),
$$
by the compactness property provided by
 Theorem~\ref{thm:p-mosco} we can also assume that $\tilde f_i^+$ $L^p$-strongly converge for all $p>1$ to a 
nonnegative $g\in \bigcap_{p>1} H^{1,p}(X,\dist,\meas)$, that 
$\tilde f_i^-$ $L^p$-strongly converge for all $p>1$ to a nonnegative $h\in \bigcap_{p>1}H^{1,p}(X,\dist,\meas)$,
so that $\tilde f_i$ strongly $L^p$-converge  for all $p>1$ to 
$f=g-h$. For $p>1$ fixed, passing to the limit as $i\to\infty$ in the equality
$$
\|\tilde f_i^+\|^p_{L^p(X,\meas_i)}+\|\tilde f_i^+\|^p_{L^p(X,\meas_i)}=\|\tilde f_i\|^p_{L^p(X,\meas_i)}
$$
we obtain that $g=f^+$ and $h=f^-$. We now claim that both $f^+$ and $f^-$ have unit $L^\infty$ norm.
The proof of the upper bound is a simple consequence of the inequalities 
$\|\tilde f^\pm_i\|_{L^p(X,\meas_i)}\leq \|\tilde f_i^\pm\|_{L^{p_i}(X,\meas_i)}=1$ for 
$p_i\geq p$, by letting first $i\to\infty$ and then $p\to\infty$, while the proof of the lower bound
is a direct consequence of \eqref{eq:extra_assumption}.

%
Theorem~\ref{thm:p-mosco} and the inequality (actually, as we already remarked, equality holds under our curvature 
assumption, see \cite{GigliHan}) between $p$-minimal relaxed slope and $2$-minimal relaxed slope $|\nabla f|$ give 
$$
\|\nabla f^\pm\|_{L^p(X,\meas)}\leq \bigl(p\Ch_p(f^\pm)\bigr)^{1/p} \le \liminf_{i \to \infty}\bigl(p_i\Ch^i_p(f_i^\pm)\bigr)^{1/p_i}
$$
for any $p\geq 2$, thus letting $p \to \infty$ gives
$$
\||\nabla f^\pm|\|_{L^\infty(X,\meas)} \le \lim_{i \to \infty}\left( \lambda_{1, p_i}(X,\dist,\meas_i)\right)^{1/p_i}.
$$
Therefore $f^\pm$ have Lipschitz representatives, still denoted by $f^\pm$, with Lipschitz constants at most the right hand side above.
The relatively open subsets $\Omega^\pm:=\{f^\pm>0\}\cap\supp\meas$ 
of $\supp\meas$ are disjoint and nonempty. Let 
$$
r(\Omega^\pm):= \sup_{x \in \Omega^\pm}\left( \inf_{y \in \partial\Omega^\pm\cap\supp(\meas)}\dist(x, y)\right).
$$
Using the inequality $r(\Omega^+)+r(\Omega^-)\leq\mathrm{diam}(\supp\meas)$, ensured by the length property
of $(\supp\meas,\dist)$, we get
\begin{equation}\label{eq:rad1}
\frac{2}{\mathrm{diam}\,\supp(\meas)} \le \max \left\{ \frac{1}{r(\Omega^+)}, \frac{1}{r(\Omega^-)}\right\}.
\end{equation}

For $\delta\in (0,1)$, take points $x^\pm\in \Omega^\pm$ with $f^\pm(x^\pm)\geq 1-\delta$, and take points $
y^{\pm} \in \partial \Omega^{\pm}\cap\supp\meas$; since $f^{\pm}(y^{\pm})=0$, we have
$$
1-\delta\leq |f^\pm(x^\pm)-f^\pm(y^\pm)| \le {\rm Lip}(f^\pm)\dist (x^{\pm}, y^{\pm}), 
$$
so that $\|f^\pm\|_{L^\infty(X,\meas)}=1$ and the arbitrariness of $y^\pm$ give 
$$
1\leq {\rm Lip}(f^\pm) r(\Omega^\pm)\le \liminf_{i \to \infty}\left( \lambda_{1, p_i}(X,\dist,\meas_i)\right)^{1/p_i}\cdot r(\Omega^\pm).
$$
Thus 
\begin{equation}\label{eq:rad2}
\max \left\{ \frac{1}{r(\Omega^+)}, \frac{1}{r(\Omega^-)}\right\} \le  \liminf_{i \to \infty}\left( \lambda_{1, p_i}(X,\dist,\meas_i)\right)^{1/p_i}
\end{equation}
and \eqref{eq:rad1} and \eqref{eq:rad2} yield the $\liminf$ inequality in \eqref{eq:upp diam}.
\end{proof}

\section{Stability of Hessians and Ricci tensor}\label{sec:10}

Recall that derivations, according to \cite{Gigli} (the definitions being inspired by \cite{Weaver}), are linear functionals
$\bb:H^{1,2}(X,\dist,\meas)\to L^0(X,\meas)$ satisfying the quantitative locality property
$$
|\bb(u)|\leq h|\nabla u|\qquad\text{$\meas$-a.e. in $X$, for all $u\in H^{1,2}(X,\dist,\meas)$}
$$
for some $h\in L^0(X,\meas)$. The minimal $h$, up to $\meas$-negligible sets, is denoted $|\bb|$. The simplest example
of derivation is the gradient derivation $\bb_v(u):=\langle \nabla v,\nabla u\rangle$ induced by $v\in H^{1,2}(X,\dist,\meas)$, 
which satisfies $|\bb_v|=|\nabla v|$ $\meas$-a.e. in $X$. By a nice duality argument, it has also been proved in \cite[Section~2.3.1]{Gigli} that the 
$L^\infty(X,\meas)$-module generated by gradient derivations is dense in the class of $L^2$ derivations. In the language of
\cite{Gigli}, $L^2$-derivations correspond to $L^2$-sections of the tangent bundle $T(X,\dist,\meas)$, viewed as dual of the
$L^2$-sections of cotangent bundle $T^*(X,\dist,\meas)$ (the latter built starting from differentials of Sobolev functions), see
\cite[Section~2.3]{Gigli} for more details. 

Even though higher order tensors will not play a big role in this paper, except for the Hessians, let us describe the basic ingredients
of the theory developed for this purpose in \cite{Gigli}.
In a metric measure space $(X,\dist,\meas)$, for $p\in [1,\infty]$ let $L^p(T^r_s(X,\dist,\meas))$ denote the space of $L^p$-tensor fields of type $(r, s)$ on 
$(X,\dist,\meas)$, defined as in \cite{Gigli}. A tensor field of type $(r,s)$ is a $L^\infty(X,\meas)$-multilinear map
$$
T:\bigotimes_{k=1}^r T(X,\dist,\meas)\otimes\bigotimes_{k=r+1}^{r+s} T^*(X,\dist,\meas)\to L^0(X,\meas)
$$
satisfying, for some $g\in L^0(X,\meas)$ a continuity property
$$
|T(u\otimes v)|\leq g|u\otimes v|_{HS}\qquad\text{$\meas$-a.e. in $X$.}
$$
with respect to a suitable Hilbert-Schmidt norm on the tensor products. The minimal (up to $\meas$-negligible sets) $g$ is denoted $|T|$ and $L^p$
tensor fields correspond to tensor fields satisying $|T|\in L^p(X,\meas)$.

In particular derivations correspond to $(0,1)$-tensor fields. We recall the following facts and definitions:

\noindent (1) any choice of $g^0,\ldots,g^{r+s}\in  W^{1,2}(X,\dist,\meas)$ induces
a product tensor field $T$ acting as follows
$$
\langle T,\bigotimes_{k=1}^r \nabla f^k \otimes\bigotimes_{k=r+1}^{r+s}df^k \rangle=
g_0\perone_{k=1}^r\bb_{f^k}(g^k)\cdot\perone_{k=r+1}^{r+s}\bb_{g^k}(f^k)
$$
and denoted $g^0\bigotimes_{1}^r d g^k \otimes\bigotimes_{r+1}^{r+s}\nabla g^k$.
Since derivations correspond to $(0,1)$-tensor fields, we recover in particular the concept of gradient derivations.

\noindent 
(2) Denoting, as in \cite{Savare}, \cite{Gigli} (recall that $D(\Delta)$ is defined as in \eqref{eq:defDelta})
$$
\mathrm{Test}F(X,\dist,\meas):=\left\{f\in \Lipb(X)\cap D(\Delta):\ \Delta f\in H^{1,2}(X,\dist,\meas)\right\},
$$
the space of finite combinations of tensor products
\[
ST^r_s(X, \dist,\meas):=\left\{ \sum_{j=1}^Ng^{j, 0}\bigotimes_{k=1}^rd g^{j ,k} \otimes \bigotimes_{k=r+1}^{r+s}\nabla g^{j,k}:\
N\geq 1,\,\,g^{j, i} \in \mathrm{Test}F(X,\dist,\meas)\right\}
\]
is dense in $L^p(T^r_s(X,\dist,\meas))$ for $p\in [1,\infty)$. This is due to the fact that the very definition of tensor product involves a completion procedure
of the class of finite sums of elementary products. Notice also that $h_t$ maps $\Lipb(X)$ into $\mathrm{Test}F(X,\dist,\meas)$
for all $t>0$.

\noindent
(3) If $(X,\dist,\meas)$ is a $RCD(K,\infty)$ space, the space $W^{2,2}(X,\dist,\meas)$ is defined in \cite{Gigli} to be the space
of all functions $f\in H^{1,2}(X,\dist,\meas)$ such that  
\begin{eqnarray}\label{eq:def_Hessian} 
2\int_X\phi\mathrm{Hess}(f)(dg\otimes dh)&=&-\int_X
\langle\nabla f,\nabla g\rangle\div(\phi\nabla h)\di\meas
-\int_X\langle\nabla f,\nabla h\rangle\div(\phi\nabla g)\di\meas\nonumber
\\&-&\int_X\phi\bigl\langle\nabla f,\nabla\langle\nabla g,\nabla h\rangle\bigr\rangle\di\meas
\end{eqnarray}
for $\phi,\,f,\,g\in \mathrm{Test}F(X,\dist,\meas)$,
with $\mathrm{Hess}(f)$ a $(0,2)$ tensor field in $L^2$. This is a Hilbert space, when endowed with the norm
$$
\|f\|_{W^{2,2}(X,\dist,\meas)}:=\left(\|f\|^2_{H^{1,2}(X,\dist,\meas)}+\| |\mathrm{Hess}(f)|\|^2_{L^2(X,\meas)}\right)^{1/2}.
$$
It has been proved in \cite[Corollary~3.3.9]{Gigli} that $H^{1,2}(X,\dist,\meas)\cap D(\Delta)\subset W^{2,2}(X,\dist,\meas)$, with
\begin{equation}\label{eq:bound_on_Gamma3}
\int_X|\mathrm{Hess}(f)|^2\di\meas\leq\int_X (\Delta f)^2+K^-|\nabla f|^2\di\meas.
\end{equation}
Notice that \eqref{eq:def_Hessian} makes sense because of \eqref{eq:bound_on_Gamma2}; on the other hand, as soon as
$f\in W^{2,2}(X,\dist,\meas)$, by approximation the formula extends from $\phi\in \mathrm{Test}F(X,\dist,\meas)$ to
$\phi\in\Lipb(X)$. In particular, in our convergence results we shall use the choice $\phi\in h_{\setQ_+}\Algebra_\bs$,
where $h$ is the semigroup relative to the limit metric measure structure. Also, arguing as in \cite[Theorem~3.3.2(iv)]{Gigli}, we immediately obtain
that, given $f\in H^{1,2}(X,\dist,\meas)$, $f\in W^{2,2}(X,\dist,\meas)$ if and only if there is $h\in L^2(X,\meas)$ satisfying
\begin{eqnarray}\label{eq:CharaW22}
&&\biggl| \sum_k \biggl(
- \int_X\langle\nabla f,\nabla g_k\rangle\div(\phi_k\psi_k\nabla h_k)\di\meas
-\int_X\langle\nabla f,\nabla h_k\rangle\div(\phi_k\psi_k\nabla g_k)\di\meas\nonumber\\
&-&\int_X\phi_k\psi_k\bigl\langle\nabla f,\nabla\langle\nabla g_k,\nabla h_k\rangle\bigr\rangle\biggr)\biggr|
\leq\int_X h|\sum_k\phi_k\psi_k\nabla g_k\otimes\nabla h_k|\di\meas
\end{eqnarray}
for any finite collection of $\phi_k,\,\psi_k\in h_{\setQ_+}\Algebra_\bs$, $g_k,\,h_k\in\mathrm{Test}F(X,\dist,\meas)$.
In addition, the smallest $h$ up to $\meas$-negligible sets is precisely $|\mathrm{Hess}(f)|$.

In the sequel we shall also use the simplified notation $\mathrm{Hess}(f)(g,h)$. 

\begin{remark}\label{rem:HS_upper_sci}
If we have finitely many $g_k,\,h_k\in H^{1,2}(X,\dist,\meas)$ and $g^k_i,\,h^k_i\in H^{1,2}(X,\dist,\meas_i)$ 
are strongly convergent to $g_k,\,h_k$ in $H^{1,2}$ and uniformly Lipschitz, then
\begin{equation}\label{eq:HS_upper_sci}
|\sum_k \phi_k\nabla g^k_i\otimes\nabla h^k_i|_i
\quad\text{$L^2$-strongly converge to}\quad
|\sum_k \phi_k\nabla g^k\otimes\nabla h^k|
\end{equation}
for any choice of $\phi_k\in\Cb(X)$. Indeed, we can use the identity 
$$
|\sum_k \phi_k\nabla g^k_i\otimes\nabla h^k_i|_i^2=
\sum_{k,\,l}\phi_k\phi_l\langle\nabla g^k_i,\nabla g^l_i\rangle_i\langle\nabla h^k_i,\nabla h^l_i\rangle_i
$$
and Theorem~\ref{thm:cont_reco}(c) which provides the $L^1$-strong convergence of
$\langle\nabla g^k_i,\nabla g^l_i\rangle_i$ to $\langle\nabla g^k,\nabla g^l\rangle$; since these gradients
are equibounded we can use Proposition~\ref{prop:strocon}(a) to improve the convergence to $L^2$ (actually
any $L^p$, $p<\infty$) convergence, so that the products $L^1$-strongly converge. 
\end{remark}

Let us consider the regularization of $h_t$
\begin{equation}\label{eq:mollih1}
h_\rho f:=\int_0^\infty\rho(s)h_sf\di s,
\end{equation}
with $\rho\in C^\infty_c((0,\infty))$ convolution kernel and, when necessary, let us define $h^i_\rho$ in an analogous way. 
Since 
\begin{equation}\label{eq:mollih2}
\Delta h_\rho f=-\int_0^\infty \rho'(s)h_sf\di s\quad\text{if $f\in L^2(X,\meas)$},\quad
\Delta h_\rho f=\int_0^\infty \rho(s)h_s\Delta f\di s\quad\text{if $f\in D(\Delta)$},
\end{equation}
it is immediately seen that $h_\rho$ maps $L^2(X,\meas)$ into $\mathrm{Test}F(X,\dist,\meas)$ and
retains many properties of $h$, namely 
\begin{equation}\label{eq:regrho1}
\sup |h_\rho f|\leq \sup|f|, \qquad\Lip(h_\rho f)\leq e^{K^-\tau}\Lip(f),
\end{equation}
(with $\tau=\sup\supp\rho$) if $f$ is bounded and/or Lipschitz, and
\begin{equation}\label{eq:regrho2}
\int_X|\nabla h_\rho f|^2\di\meas\leq\int_X|\nabla f|^2\di\meas\qquad\text{if $f\in H^{1,2}(X,\dist,\meas)$,}
\end{equation}
\begin{equation}\label{eq:regrho3}
\int_X|\Delta h_\rho f|^2\di\meas\leq\int_X|\Delta f|^2\di\meas\qquad\text{if $f\in D(\Delta)$}.
\end{equation}
Then, we define
\begin{eqnarray*}
\mathrm{Test}_*F(X,\dist,\meas)&:=&\left\{h_\rho\bigl( L^2\cap L^\infty(X,\meas)\bigr):\ \text{$\rho\in C^\infty_c((0,\infty))$ convolution kernel}\right\}
\\&& \subset\mathrm{Test}F(X,\dist,\meas).
\end{eqnarray*}
By letting $\rho\to\delta_0$ it is immediately seen from \eqref{eq:regrho1}, \eqref{eq:regrho2}, \eqref{eq:regrho3} 
that the class $\mathrm{Test}_*F(X,\dist,\meas)$ is dense in $\mathrm{Test}F(X,\dist,\meas)$, namely for any 
$f\in \mathrm{Test}F(X,\dist,\meas)$ there exist $f_n\in \mathrm{Test}_*F(X,\dist,\meas)$
strongly convergent in $H^{1,2}$ to $f$, with $\sup|f_n|\leq\sup|f|$, $\Lip(f_n)\leq\Lip f$, and
$\Delta f_n\to\Delta f$ strongly in $H^{1,2}$.

In the next proposition we show a canonical
approximation of test functions in the class $\mathrm{Test}F(X,\dist,\meas)$ by test functions for the approximating metric measure
structures. Notice that we don't know if condition (b) can be improved, getting strong $H^{1,2}$ convergence of $|\nabla f_i|^2_i$. A

\begin{proposition}\label{prop;ex. app. seq}
Let $f \in \mathrm{Test}F(X,\dist,\meas)$.
Then there exist $f_i\in \mathrm{Test}_*F(X,\dist,\meas_i)$ with 
$\|f_i\|_{L^\infty(X,\meas_i)}\leq\|f\|_{L^\infty(X,\meas)}$ and $\sup_i\Lip(f_i)<\infty$,
such that $f_i$ and $\Delta_i f_i$ strongly converge to 
$f$ and $\Delta f$ in $H^{1, 2}$, respectively.
Moreover, these properties yield:
\begin{itemize}
\item[(a)]  $|\nabla f_i|_i^2$ $L^1$-strongly and $L^2_\loc$-strongly converge to $|\nabla f|^2$;
\item[(b)] $|\nabla f_i|_i^2$ weakly converge to $|\nabla f|^2$ in $H^{1,2}$. 
\end{itemize}
\end{proposition}
\begin{proof} Let us assume first that $f=h_\rho g$ for some $g\in L^2\cap L^\infty(X,\meas)$ and some convolution kernel $\rho$. We define
$f_i$ as $h^i_\rho g_i$, with $g_i$ $L^2$-strongly convergent to $g$, with $\|g_i\|_{L^\infty(X,\meas_i)}\leq\|g\|_{L^\infty(X,\meas)}$. 
It is clear from the construction that $\|f_i\|_{L^\infty(X,\meas_i)}\leq\|f\|_{L^\infty(X,\meas)}$ and that $\sup_i\Lip(f_i)<\infty$.
From \eqref{eq:Brezis1} and \eqref{eq:Brezis2},
together with the first formula in \eqref{eq:mollih2} (applied to $h^i_\rho$), we obtain that both $f_i$ and $\Delta_i f_i$ are bounded in $H^{1,2}$, 
and their strong convergence is a direct consequence of Corollary~\ref{cor:strocoh}(b) and of  
\eqref{eq:mollih2} again.

The weak convergence in $H^{1,2}$ 
of $|\nabla f_i|^2_i$ to $|\nabla f|^2$ follows by the apriori estimates \eqref{eq:bound_on_Gamma1} and 
\eqref{eq:bound_on_Gamma2}, that ensure the uniform bounds in $H^{1,2}$, and by Theorem~\ref{thm:cont_reco}(c)
that identifies the $L^1$-strong limit (and therefore the weak $H^{1,2}$ limit) as $|\nabla f|^2$. Theorem~\ref{thm:gms13} provides
the relative compactness in $L^2_\loc$ of $|\nabla f_i|_i^2$ and then proves $L^2_\loc$-convergence of
$|\nabla f_i|_i^2$ to $|\nabla f|^2$ as well.

When $f\in\mathrm{Test}F(X,\dist,\meas)$ we apply the previous approximation procedure to $h_\rho f$ and then we make a diagonal
argument, letting $\rho\to \delta_0$, noticing that the first identity in \eqref{eq:mollih2} grants the strong convergence in $H^{1,2}$ of 
$\Delta_i h_\rho^i f_i$ to $\Delta h_\rho f$, while the second identity in \eqref{eq:mollih2}  grants
$$
\|\Delta h_\rho f\|_{L^2(X,\meas)}\leq \|\Delta f\|_{L^2(X,\meas)},\qquad
\|\nabla \Delta h_\rho f\|_{L^2(X,\meas)}\leq \|\nabla \Delta f\|_{L^2(X,\meas)}.
$$
\end{proof}

\begin{theorem}[Stability of $W^{2,2}$ regularity and weak convergence of Hessians]\label{thm;stability hessian}
Let $f_i \in W^{2, 2}(X,\dist,\meas_i)$ with $\sup_i\|f_i\|_{W^{2, 2}(X,\dist,\meas_i)}<\infty$, and assume that $f_i$ strongly converge in $H^{1,2}$ to 
$f \in H^{1, 2}(X,\dist,\meas)$.\\
Then $f \in W^{2, 2}(X,\dist,\meas)$ and $\mathrm{Hess}_i(f_i)$ $L^2$-weakly converge to $\mathrm{Hess}(f)$
in the following sense: whenever $g_i\in H^{1,2}(X,\dist,\meas_i)$ 
are uniformly Lipschitz and strongly converge in $H^{1,2}$ to
$g\in H^{1,2}(X,\dist,\meas)$, 
$$\text{$\mathrm{Hess}_i(f_i)(g_i,g_i)$ $L^2$-weakly converge to $\mathrm{Hess}(f)(g,g)$.}$$

In addition, $|\mathrm{Hess}(f)|\leq H$ $\meas$-a.e. for any
$L^2$-weak limit point $H$ of $|\mathrm{Hess}_i(f_i)|$, and in particular
\begin{equation}\label{eq:lsc_Hessians}
\int_X|\mathrm{Hess}(f)|^2\di\meas\leq\liminf_{i\to\infty}\int_X|\mathrm{Hess}_i(f_i)|^2\di\meas_i.
\end{equation}
\end{theorem}
\begin{proof} Let $g\in \mathrm{Test}F(X,\dist,\meas)$ and let
$H$ be a $L^2$-weak limit point of $|\mathrm{Hess}_i(f_i)|$. Let $(g_i)$ 
be provided by Proposition~\ref{prop;ex. app. seq}. We will first prove convergence of the Hessians
under these stronger convergence assumption on $g_i$.

In order to identify the $L^2$-weak limit of
$\mathrm{Hess}(f_i)(g_i,g_i)$ we want to pass to the limit as $i\to\infty$ in the expression
$$
-2\int_X\langle\nabla f_i,\nabla g_i\rangle_i\div(\phi\nabla g_i)\di\meas_i-\int_X\phi\langle\nabla f_i,\nabla |\nabla g_i|_i^2\rangle_i\di\meas_i
$$
with $\phi\in h_{\setQ_+}\Algebra_\bs$.
Let us analyze the first term. Since $\div(\phi\nabla g_i)=\phi\Delta_i g_i+\langle\nabla g_i,\nabla\phi\rangle$, this term $L^2$-strongly converges
to $\div(\phi\nabla g)= \phi\Delta g+\langle\nabla g,\nabla\phi\rangle$. On the other hand, by Theorem~\ref{thm:cont_reco}(b), the term 
$\langle\nabla f_i,\nabla g_i\rangle_i$ $L^2$-weakly converges to $\langle\nabla f,\nabla g\rangle$. 
This proves the convergence of the first term.

Let us analyze the second term. Since Proposition~\ref{prop;ex. app. seq}(b) shows that
$|\nabla g_i|_i^2$ weakly converge in $H^{1,2}$ to $|\nabla g|^2$, we can apply Theorem~\ref{thm:cont_reco}(b) again
to obtain the convergence of $\int_X\phi\langle\nabla f_i, \nabla |\nabla g_i|_i^2\rangle_i\di\meas_i$ to 
$\int_X\phi\langle\nabla f,\nabla |\nabla g|^2\rangle\di\meas$.

This completes the proof under the additional assumption on $g_i$. 
In the general case it is sufficient to apply the
already proved convergence result to $h^i_\rho g_i$, with $\rho$ convolution kernel with support
in $(0,\infty)$, noticing the uniform Lipschitz bound on $g_i$
yields
\begin{eqnarray*}
\int_X|\mathrm{Hess}(f_i)(g_i,g_i)-\mathrm{Hess}(f_i)(h^i_\rho g_i,h^i_\rho g_i)|\di\meas_i&\leq&
\int_X |\mathrm{Hess}(f_i)||\nabla g_i\otimes\nabla g_i-\nabla h_\rho^i g_i\otimes\nabla h_\rho^ig_i|\di\meas_i\\
&\leq&C\int_X |\mathrm{Hess}(f_i)||\nabla g_i-\nabla h_\rho^i g_i|_i\di\meas_i
\end{eqnarray*}
and that the strong $H^{1,2}$ convergence of $h^i_\rho g_i$ to $h_\rho g$ yields
$$
\lim_{\rho\to \delta_0}\limsup_{i\to\infty} \int_X|\nabla g_i-\nabla h_\rho^i g_i|_i^2\di\meas_i=0.
$$

The inequality $|\mathrm{Hess}(f)|\leq H$ can be proved as follows. We start from the observation that, by bilinearity,
$$
-\int_X\langle\nabla f_i,\nabla g_i\rangle_i\div(\phi\psi\nabla h_i)\di\meas_i-
\int_X\langle\nabla f_i,\nabla h_i\rangle_i\div(\phi\psi\nabla g_i)\di\meas_i-
\int_X\phi\psi\bigl\langle\nabla f_i,\nabla\langle\nabla g_i,\nabla h_i\rangle_i\bigr\rangle_i\di\meas_i
$$
converges to
$$
-\int_X\langle\nabla f,\nabla g\rangle\div(\phi\psi\nabla h)\di\meas-
\int_X\langle\nabla f,\nabla h\rangle\div(\phi\psi\nabla g)\di\meas-
\int_X\phi\psi\bigl\langle\nabla f,\nabla\langle\nabla g,\nabla h\rangle\bigr\rangle\di\meas
$$
for any $\phi,\,\psi\in h_{\setQ_+}\Algebra_\bs$ whenever $g_i,\,h_i\in \mathrm{Test}F(X,\dist,\meas_i)$ are uniformly Lipschitz and
strongly converge in $H^{1,2}$ to $g,\,h\in \mathrm{Test}F(X,\dist,\meas)$ respectively.
This, taking also Remark~\ref{rem:HS_upper_sci} into account, enables to pass to the limit in 
\eqref{eq:CharaW22} written for $f_i$, to get
\begin{eqnarray*}
&&\biggl| \sum_k \biggl(
- \int_X\langle\nabla f,\nabla g_k\rangle\div(\phi_k\psi_k\nabla h_k)\di\meas
-\int_X\langle\nabla f,\nabla h_k\rangle\div(\phi_k\psi_k\nabla g_k)\di\meas\nonumber\\
&-&\int_X\phi_k\psi_k\bigl\langle\nabla f,\nabla\langle\nabla g_k,\nabla h_k\rangle\bigr\rangle\biggr)\biggr|
\leq\int_X H|\sum_k\phi_k\psi_k\nabla g_k\otimes\nabla h_k|\di\meas
\end{eqnarray*}
for any finite collection of $\phi_k,\,\psi_k\in h_{\setQ_+}\Algebra_\bs$, $g_k,\,h_k\in\mathrm{Test}F(X,\dist,\meas)$.
This proves that $|\mathrm{Hess}(f)|\leq H$ $\meas$-a.e. in $X$.
\end{proof}

In the next corollary we use the bounds on laplacians of $f_i$ to obtain at the same time strong convergence
in $H^{1,2}$ and the uniform bound in $W^{2,2}$, so that the conclusions of Theorem~\ref{thm;stability hessian} apply.

\begin{corollary}[Weak stability of Hessians under Laplacian bounds]\label{thm;compactness lap}
Let $f_i\in D(\Delta_i)$ with $$\sup_i (\|f_i\|_{L^2(X,\meas_i)}+\|\Delta_i f_i\|_{L^2(X,\meas_i)})<\infty$$
and assume that $f_i$ $L^2$-strongly converge to $f$. Then $f \in D(\Delta )$ and
\begin{itemize}
\item[(i)] $f_i$ strongly converge to $f$ in $H^{1, 2}$;
\item[(ii)]   $\Delta_i f_i$ $L^2$-weakly converge to $\Delta f$; 
\item[(iiii)] the Hessians of $f_i$ are weakly convergent to the Hessian of $f$ as in Theorem~\ref{thm;stability hessian}.
\end{itemize}
\end{corollary}
\begin{proof}
Statements (i) and (ii) follows by By Corollary~\ref{cor:strocoh}(a), while statement (iii) is a consequence of
Theorem~\ref{thm;stability hessian} and of \eqref{eq:bound_on_Gamma3}.
\end{proof}

In the final part of his work \cite{Gigli}, motivated also by the measure-valued $\Gamma_2$ operator
introduced in \cite{Savare}, Gigli introduced a weak Ricci tensor $\mathbf{Ric}$. It is a sort of measure-valued $(0,2)$-tensor, 
whose action on gradients of functions $f\in \mathrm{Test}F(X,\dist,\meas)$ is given by
\begin{equation}\label{eq:defRic}
\mathbf{Ric}(\nabla f,\nabla f):={\bf\Delta}\frac 12|\nabla f|^2-|\mathrm{Hess}(f)|^2\meas-\langle\nabla f,\nabla\Delta f\rangle\meas,
\end{equation}
where the potentially singular part w.r.t. $\meas$ comes from the distributional laplacian ${\bf\Delta}$. The measure
defined in \eqref{eq:defRic} is bounded from below by $K|\nabla f|^2\meas$ and it is a capacitary measure, namely
it vanishes on sets with null capacity (with respect to the Dirichlet form associated to $\Ch$), hence its duality
with functions in $H^{1,2}(X,\dist,\meas)$ is well-defined.

Actually, $\mathbf{Ric}$ can be defined as a bilinear form on a larger class  
$H^{1,2}_H(T(X,\dist,\meas))$ of vector fields, weakly differentiable in a suitable sense, 
which includes gradient vector fields of functions in $\mathrm{Test}F(X,\dist,\meas)$; on the other hand, using 
the linearity property of Proposition~3.6.9 in \cite{Gigli}, as well as the 
continuity property (3.6.13) of Theorem~3.6.7, one can prove that \eqref{eq:Ricci2} holds iff
$\mathbf{Ric}(v,v)\geq \zeta |v|^2$ for all $v\in H^{1,2}_H(T(X,\dist,\meas))$. For this reason
we confine ourselves to the smaller class of vector fields.
 
Using the tools developed so far we are able to prove a kind of upper semicontinuity, in the measure-valued
sense, for $\mathbf{Ric}$ under measured Gromov-Hausdorff convergence. 

\begin{theorem}[Upper semicontinuity of Ricci curvature]\label{thm; upp ricci}
Assume that $(X,\dist,\meas_i)$ are $RCD(K_i,\infty)$ spaces satisfying
\begin{equation}\label{eq:Ricci1}
\mathbf{Ric}_i(\nabla f,\nabla f)\geq \zeta|\nabla f|_i^2\qquad\forall f\in \mathrm{Test}F(X,\dist,\meas_i)
\end{equation}
for some $\zeta\in C(X)$ with $\zeta^-$ bounded. Then 
\begin{equation}\label{eq:Ricci2}
\mathbf{Ric}(\nabla f,\nabla f)\geq \zeta|\nabla f|^2\qquad\forall f\in \mathrm{Test}F(X,\dist,\meas).
\end{equation}
\end{theorem}
\begin{proof} Setting $K=\sup\zeta^-$, from \eqref{eq:Ricci1} and from the characterization of $RCD(K,\infty)$ spaces based on
Bochner's inequality in \cite{AmbrosioGigliSavare15} we obtain that $(X,\dist,\meas_i)$ are $RCD(K,\infty)$ spaces.
By a truncation argument, is not restrictive to assume that $\zeta\in C_b(X)$. Assume that $f\in \mathrm{Test}F(X,\dist,\meas)$ and
let $f_i\in\mathrm{Test}F(X,\dist,\meas_i)$ be strongly convergent in $H^{1,2}$ to $f$, with 
$\sup_i(\sup_X|f_i|+\Lip(f_i))<\infty$, $\Delta_i f_i$ strongly convergent
to $\Delta f$ in $H^{1,2}$ and $|\nabla f_i|_i^2$ weakly convergent in $H^{1,2}$ to $|\nabla f|^2$, 
whose existence is granted by Proposition~\ref{prop;ex. app. seq}. 

We want to pass to the limit
as $i\to\infty$ in the integral formulation
\begin{eqnarray}\label{eq:Ricci3}
&&-\frac 12\int_X \langle\nabla\phi_i,\nabla |\nabla f_i|_i^2\rangle_i\di\meas_i-
\int_X\phi_i|\mathrm{Hess}(f_i)|_i^2\di\meas_i-
\int_X\phi_i\langle\nabla f_i,\nabla\Delta_i f_i\rangle_i\di\meas_i\nonumber\\
&\geq&\int_X\zeta\phi_i|\nabla f_i|^2_i\di\meas_i
\end{eqnarray}
of \eqref{eq:Ricci1}, with $\phi_i\in H^{1,2}(X,\dist,\meas_i)$ bounded and nonnegative, thus getting the integral formulation of \eqref{eq:Ricci2}. 
To this aim, for $\phi\in H^{1,2}(X,\dist,\meas)$, let $\phi_i$ be uniformly bounded, nonnegative and strongly convergent in $H^{1,2}$ to $\phi$.
First of all, since $|\nabla f_i|_i^2$ $L^1$-strongly converge to $|\nabla f|^2$, 
the right hand sides converge to $\int_X\zeta\phi|\nabla f|^2\di\meas$. Also the convergence
of the third term in the left hand side to $\int_X\phi\langle\nabla f,\nabla\Delta f\rangle\di\meas$
is ensured by Theorem~\ref{thm:cont_reco}(b). In order to handle the first term, we just use 
\eqref{eq:continuitymean}. Finally, in connection
with the Hessians, possibly extracting a subsequence we obtain a $L^2$-weak limit point $H$ of $|\mathrm{Hess}_i(f_i)|$, with
$H\geq |\mathrm{Hess}(f)|$ $\meas$-a.e. in $X$.

Summing up, passing to the limit as $i\to\infty$ in \eqref{eq:Ricci3} one obtains the inequality
$$
-\frac 12\int_X \langle\nabla\phi,\nabla |\nabla f|^2\rangle\di\meas- \int_X\phi H^2\di\meas-
\int_X\phi\langle\nabla f,\nabla\Delta f\rangle\di\meas
\geq\int_X\zeta\phi|\nabla f|^2\di\meas.
$$
Using the inequality $H\geq |\mathrm{Hess}(f)|$ $\meas$-a.e. in $X$ we conclude the proof.
\end{proof}

\begin{remark}
For any $r \in (0, 1)$, it is easy to construct a sequence $({\bf g}_i^r)$ of Riemannian metrics on $\mathbf{S}^2$ with 
sectional curvature bounded below by $1$ 
such that $(\mathbf{S}^2, {\bf g}_i^r) \to  [0, \pi] \times _{\sin}\mathbf{S}^1(r)$ in the Gromov-Hausdorff sense, where $\mathbf{S}^1(r):=\{x \in \mathbf{R}^2; |x|=r\}$ (the limit space is an Alexandrov space of curvature $\ge 1$).
Note that  $[0, \pi] \times _{\sin}\mathbf{S}^1(r) \to [0, \pi]$ as $r \to 0$ in the Gromov-Hausdorff sense, and that
\[\frac{\haus^2(B_s(x_0))}{\haus^2( [0, \pi] \times _{\sin}\mathbf{S}^1(r))}=\frac{\haus^2(B_s(x_{\pi}))}{\haus^2([0, \pi] \times _{\sin}\mathbf{S}^1(r))}=\frac{1}{2}\int_0^s \sin t\di t\]
for any $s \in [0, \pi]$, where $x_0 = (0, *)$ and $x_{\pi} =(\pi, *)$.
Thus, by a diagonal argument, there exist Riemannian metrics $({\bf g}_i)$ on $\mathbf{S}^2$ (in fact ${\bf g}_i:={\bf g}_i^{r_i}$ for some $r_i \to 0$) with sectional curvature bounded below by $1$ 
such that $(\mathbf{S}^2, {\bf g}_i, \haus^2/\haus^2(\mathbf{S}^2)) \to ([0, \pi],{\bf g},\upsilon )$ in the measured Gromov-Hausdorff sense, where ${\bf g}$ is the Euclidean metric and $\upsilon$ is the Borel probability measure on $[0, \pi]$  defined by
\[\upsilon ([r, s])=\frac{1}{2}\int_r^s \sin t\di t\]
for any $r, s \in [0, \pi]$ with $r \le s$.
Let us consider eigenfunctions $f_i \in C^{\infty}(\mathbf{S}^2)$ of the first positive eigenvalues of $\Delta_i$ with $\|f_i\|_{L^2(\mathbf{S}^2,\meas_i)}=1$, where $\meas_i= \haus^2/\haus^2(\mathbf{S}^2)$ with respect to ${\bf g}_i$.
Then, by \cite{CheegerColding3} we can assume with no loss of generality that $f_i$ strongly converge to $f$ in $H^{1, 2}$,
with $f$ eigenfunction of the first positive eigenvalue of $\Delta$. 
It is known that $\Delta f=2f$ and that $\lim_i\||\mathrm{Hess}_i(f_i)+f_i{\bf g}_i|\|_{L^2(X,\meas_i)}=0$. 
Moreover we can prove that $f(t)=3\cos t$.
Note that these observations correspond to the Bonnet-Mayers theorem and the rigidity on singular spaces.
See \cite{CheegerColding,CheegerColding3} for the proofs.

In particular $\lim_i\||\mathrm{Hess}_i(f_i)|\|_{L^2(\mathbf{S}^2,\meas_i)}=2\lim_i\|f_i\|_{L^2(\mathbf{S}^2,\meas_i)}=2$.
On the other hand, it was proven in \cite{Honda2} that ${\bf g}_i$ $L^2$-weakly converge to
${\bf g}$ on $[0, \pi]$. Thus $\mathrm{Hess}(f)+f {\bf g}=0$ in $L^2$. In particular $\||\mathrm{Hess}(f)|\|_{L^2([0, \pi],\upsilon)}=\|f\|_{L^2([0, \pi],\upsilon)}=1$.
Thus these facts give
\[\lim_{i \to \infty}{\bf Ric}_i(\nabla f_i, \nabla f_i)(\mathbf{S}^2, {\bf g}_i, \meas_i)<{\bf Ric}(\nabla f, \nabla f)([0, \pi], {\bf g}, \upsilon ),\]
i.e. the Ricci curvatures are strictly increasing even in the case when 
$f_i,\, |\nabla f_i|_i^2,\,\Delta_i f_i$ are uniformly bounded, and strongly converge to $f,\,|\nabla f|^2,\,\Delta f$ in $H^{1, 2}$, respectively.
In this respect, Theorem~\ref{thm; upp ricci} might be sharp.
Moreover this example also tells us that, in general, 
the condition that $\Delta_i f_i$ $L^2$-strongly converge to $\Delta f$ does not imply that 
$|\mathrm{Hess}_i(f_i)|$ $L^2$-strongly converge to $|\mathrm{Hess}(f)|$.
\end{remark}

\begin{remark}\label{rem:beKNstable}
With a very similar argument one can prove stability of the $BE(K,N)$ condition 
$$
\frac{1}{2}\Delta |\nabla f|^2\geq \langle\nabla f,\nabla\Delta f\rangle+\frac{(\Delta f)^2}{N}+K|\nabla f|^2,
$$
with $K:X\to (-\infty,+\infty]$ lower semicontinuous and bounded from below, $N:X\to (0,\infty]$ upper semicontinuous. 
Notice that the strategy of passing to an integral formulation, adopted in \cite[Theorem~5.8]{AmbrosioGigliSavare15},
seems to work only when $K$ and $N$ are constant.
\end{remark}

\section{Dimensional stability results}\label{sec:11}

In this section only we state results that depend on the assumption $N<\infty$. We recall that the definition of $RCD^*(K,N)$
space has been proposed in \cite{Gigli1} and deeply investigated and characterized in various ways in \cite{ErbarKuwadaSturm} (via the so-called
Entropy power functional, a dimensional modification of Shannon's logarithmic entropy) and in \cite{AmbrosioMondinoSavare} 
(via nonlinear diffusion semigroups induced by R\'enyi's $N$-entropy), see also \cite{AmbrosioGigliSavare15} in connection with
the stability point of view. Starting from $RCD(K,\infty)$, the conditions $RCD^*(K,N)$
amounts to the following reinforcement of Bochner's inequality
\begin{equation}\label{eq:BochnerN}
{\bf\Delta} \frac 12 |\nabla f|^2\geq \frac  1 N(\Delta f)^2\meas+\langle\nabla f,\nabla\Delta f\rangle\meas+K|\nabla f|^2\meas
\end{equation}
in the class $\mathrm{Test}F(X,\dist,\meas)$.

\begin{proposition}\label{cor:unif est p-lap} There exist positive and finite constants $C_i(\alpha,N)$, $i=1,2$, such that
for any $RCD^*(K, N)$-space $(Y,\dist,\meas)$ with $\supp\meas=Y$, $\meas(Y)=1$ and finite diameter one has
\begin{equation}\label{eq:unif bound}
0<C_1(K(\mathrm{diam}\,Y)^2, N) \le \mathrm{diam}\,Y \left(\lambda_{1, p}(Y,\dist,\meas)\right)^{1/p}\le C_2(K (\mathrm{diam}\,Y)^2, N)<\infty
\end{equation}
for any $p \in [1, \infty]$.
\end{proposition}
\begin{proof}
Since the rescaled metric measure space
$$
\bigl(Y, (\mathrm{diam}\,Y)^{-1}\dist,\meas\bigr)
$$
is an $RCD^*(K(\mathrm{diam}\,Y)^2, N)$-space, and
$$
\lambda_{1, p}\bigl( Y, (\mathrm{diam}\,Y)^{-1}\dist, \meas \bigr)=\left(\mathrm{diam}\,Y\right)^p\lambda_{1, p}(Y,\dist,\meas),
$$
it suffices to check (\ref{eq:unif bound}) under $\mathrm{diam}\,Y=1$.

Let $\mathcal{M}(K,N)$ be the set of all isometry classes of $RCD^*(K, N)$ spaces $(Y,\dist,\meas)$ 
satisfying $\supp\meas=Y$, $\mathrm{diam}\,Y=1$ and $\meas(Y)=1$. 
It is known that this set is sequentially compact with respect to the measured Gromov-Hausdorff 
convergence  by \cite{AmbrosioGigliSavare15,ErbarKuwadaSturm}.
We consider the function $F$ on $\mathcal{M}(K,N) \times [1, \infty]$ defined by
$$
F\left( (Y,\dist,\meas), p\right) :=\left(\lambda_{1, p}(Y,\dist,\meas)\right)^{1/p}.
$$
Hence, Theorem~\ref{thm:p-spec} and Theorem~\ref{thm:gros} yield that $F$ is continuous.
In particular the maximum and the minimum exist.
Moreover by the definition these depend only on the parameters $N$ and $K$.
This shows \eqref{eq:unif bound}.
\end{proof}

\begin{remark}
The finiteness of $N$ in Proposition~\ref{cor:unif est p-lap} is essential, i.e. the estimate $C_1(KR^2) \le \mathrm{diam}\,Y \left(\lambda_{1, p}(Y,\dist,\meas)\right)^{1/p}\le C_2(KR^2)$ does not hold for $RCD(K, \infty)$-spaces.
Indeed, the standard $n$-dimensional unit sphere with the standard probability measure $(\mathbf{S}^n, \dist_n, \meas_n)$ satisfies 
$$\lim_{n \to \infty}\lambda_{1, 2}(\mathbf{S}^n, \dist_n, \meas_n)=\infty.$$
\end{remark}

For any $N \in (1, \infty)$ and any $p \in [1, \infty]$ let us denote $(\lambda^N_{1, p})^{1/p}$ the infimum of $(\lambda_{1, p})^{1/p}$ in the set 
$\mathcal{M}(N)$ of all isometry classes of $RCD^*(N-1, N)$ probability spaces.
For $p=2$ the sharp Poincar\'e inequality for $CD^*(N-1, N)$-spaces given in \cite{Sturm06} by Sturm yields $(\lambda_{1, 2}^N)^{1/2}=N^{1/2}$ which coincides with $(\lambda_{1, 2}(\mathbf{S}^N, \dist, \meas_N))^{1/2}$ if $N$ is an integer.
The Bonnet-Meyers theorem for $CD^*(N-1, N)$-spaces given in \cite{Sturm06} by Sturm gives $(\lambda_{1, \infty}^N)^{1/\infty}=2/\pi$ which also coincides with  $(\lambda_{1, \infty}(\mathbf{S}^N, \dist, \meas_N))^{1/\infty}$ if $N$ is an integer.

The following rigidity theorem is proven by Ketterer in \cite{Ketterer15a, Ketterer15b}.

\begin{theorem}\label{thm: LO}
For any $p \in \{2, \infty\}$, any $N \in (1, \infty)$, and any $RCD^*(N-1, N)$-space $(Y,\dist,\meas)$ with $\supp\meas=Y$, the equality
$$
\left(\lambda_{1, p}(Y,\dist,\meas)\right)^{1/p} = \bigl( \lambda_{1, p}^N\bigr)^{1/p}
$$
holds if and only if $(Y,\dist,\meas)$ is isometric to the spherical suspension of an $RCD^*(N-2, N-1)$-space.

Furthermore for any $p \in \{2, \infty \}$, any $N \in (1, \infty)$, and any $\epsilon>0$ there exists 
$\delta:=\delta(p, N, \epsilon)>0$ such that if an $RCD^*(N-1,N)$-space $(Y,\dist,\meas)$ satisfies $\supp\meas=Y$ and 
$$
\left|\left(\lambda_{1, p}(Y,\dist,\meas)\right)^{1/p} -  \bigl( \lambda_{1, p}^N\bigr)^{1/p}\right|<\delta,
$$
then 
$$
\left|\left(\lambda_{1, q}(Y,\dist,\meas)\right)^{1/q} -  \bigl( \lambda_{1, q}^N\bigr)^{1/q}\right|<\epsilon,
$$
for any $q \in \{2, \infty\}$ and there exists an $RCD^*(N-2, N-1)$-space $(Z,\rho, \nu)$ such that
$$
d_{GH}\left((Y,\dist,\meas), ([0, \pi] \times_{\sin}^{N-1}(Z,\rho,\nu) )\right)<\epsilon.
$$
\end{theorem}

The following theorem is proven by Cavalletti-Mondino in \cite{CavallettiMondino15a, CavallettiMondino15b}.

\begin{theorem}\label{thm:almost rigidity}
We have the following.
\begin{enumerate}
\item[(i)] For any $p \in [1, \infty)$ and $N \in \mathbf{N}_{\ge 2}$, we have 
$(\lambda_{1, p}^N)^{1/p}=(\lambda_{1, p}(\mathbf{S}^N, \dist_N, \meas_N))^{1/p}$.
\item[(ii)] For any $p \in [1, \infty)$, any $N \in (1, \infty)$ and any $RCD^*(N-1, N)$-space $(Y,\dist,\meas)$ with $\supp\meas=Y$, if the equality
$$
\left(\lambda_{1, p}(Y,\dist,\meas)\right)^{1/p} = \bigl( \lambda_{1, p}^N \bigr)^{1/p}.
$$
holds, then $(Y,\dist,\meas)$ is isometric to the spherical suspension of an $RCD^*(N-2, N-1)$-space.

Furthermore for any $p \in [1, \infty)$, any $N \in (1, \infty)$, and any $\epsilon>0$ there exists 
$\delta:=\delta(p, N, \epsilon)>0$ such that if an $RCD^*(N-1,N)$-space $(Y,\dist,\meas)$ satisfies $\supp\meas=Y$ and
$$
\left|\left(\lambda_{1, p}(Y,\dist,\meas)\right)^{1/p} - \bigl( \lambda_{1, p}^N\bigr)^{1/p}\right|<\delta,
$$
then $\left| \mathrm{diam}\,(Y, \dist)-\pi\right|<\epsilon$.
\end{enumerate}
\end{theorem}

We now give a model metric measure space whose $(\lambda_{1, p})^{1/p}$ attains $(\lambda_{1, p}^N)^{1/p}$ for general $N$.

\begin{proposition}\label{prop:char model sp}
For any $N \in (1, \infty)$, let $([0, \pi], \dist, \upsilon_N)$ with $\dist$ equal to the Euclidean distance and
$$
\upsilon_N(A):=\frac{1}{\int_0^\pi\sin^{N-1}t\di t}\int_A\sin^{N-1}t\di t.
$$ 
Then $([0, \pi], \dist, \upsilon_N)$ is an $RCD^*(N-1, N)$-space with 
$$
\bigl( \lambda_{1, p}([0, \pi], \dist, \upsilon_N)\bigr)^{1/p}=(\lambda_{1, p}^N)^{1/p}\qquad\forall p\in [1,\infty].
$$
\end{proposition}
\begin{proof}
By \cite[Theorem 1.4]{CavallettiMondino15b}, for any $p \in [1, \infty]$, $(\lambda_{1, p}^N)^{1/p}$ coincides with
the infimum in the smaller class
\begin{equation}\label{eq:1-dim}
\inf \left\{ \lambda_{1, p}([0, \pi], \dist, \meas); ([0, \pi], \dist, \meas) \in \mathcal{M}(N)\right\}.
\end{equation}
By Theorem \ref{thm:p-spec} and the sequencial compactness of $\mathcal{M}(N)$, there exists a 
Borel probability measure $\meas^p$ on $[0, \pi]$ such that $(\lambda_{1, p}([0, \pi], \dist, \meas^p))^{1/p}=(\lambda_{1, p}^N)^{1/p}$.
Then the maximal diameter theorem and $p$-Obata theorem for general $N \in (1, \infty)$ yield $\meas^p=\upsilon_N$.
This completes the proof.
\end{proof}

As a corollary of Theorem~\ref{thm:p-spec} and Theorem~\ref{thm:gros},  we have a 
generalization of Theorem~\ref{thm: LO} and Theorem~\ref{thm:almost rigidity} as follows.
It is worth pointing out that this is new even in the class of smooth metric measure spaces, and shows that 
the parameter $\delta$ in Theorem~\ref{thm:almost rigidity} can be chosen independent of $p$:

\begin{corollary}\label{cor;shperical suspension}
For any $N \in (1, \infty)$ and any $\epsilon>0$ there exists $\delta:=\delta (N, \epsilon)>0$ such that if 
an $RCD^*(N-1, N)$ space $(X,\dist,\meas)$ satisfies $\supp\meas=X$, $\meas(X)=1$ and
$$
\left| \left( \lambda_{1, p}(X,\dist,\meas) \right)^{1/p}-\bigl( \lambda_{1, p}^N\bigr)^{1/p} \right|<\delta
$$
for some $p \in [1, \infty]$, then
$$
\left| \left( \lambda_{1, q}(X,\dist,\meas) \right)^{1/q}-\bigl( \lambda_{1, q}^N\bigr)^{1/q} \right|<\epsilon
$$
for all $q \in [1, \infty]$.
\end{corollary}
\begin{proof}
We first prove that if an $RCD^*(N-1, N)$-space $(Y, \dist, \meas)$ satisfies $\supp\meas=Y$ and $\mathrm{diam}\,(Y, \dist)=\pi$, then 
$(\lambda_{1, p}(Y, \dist, \meas))^{1/p}=(\lambda_{1, p}^N)^{1/p}$ for any $p \in [1, \infty)$.

By Theorem \ref{thm: LO}, there exists an $RCD^*(N-2, N-1)$-space $(Z,\rho, \nu)$ such that
$(Y,\dist,\meas)$ is isometric to $([0, \pi] \times_{\sin}^{N-1}(Z,\rho,\nu) )$ and, from now on, we
make this identification. Note that
for any $f \in L^1([0, \pi], \upsilon_N)$ the function $f_0(y):=f(t)$ for $y=(t, z)$ is in $L^1(Y, \dist, \meas)$, 
and satisfies $c_p(f_0)=c_p(f), \|f_0\|_{L^p}=\|f\|_{L^p}$ for any $f \in L^p([0, \pi], \upsilon_N)$. In addition
\begin{equation}\label{eq:coarea2}
\int_Yf_0\di\meas = \int_0^{\pi}f\di\upsilon_N.
\end{equation}
Let $g \in \Lip([0, \pi], \dist)$ with $c_p(g)=\|g\|_{L^p}=1$ (with respect to $\upsilon_N$).
Using the agreement of minimal relaxed slope with local Lipschitz constant in metric measure spaces satisfying
the doubling and $(1,p)$-Poincar\'e condition (first proved in \cite{Cheeger}, see also \cite{AmbrosioColomboDiMarino}), it is easy to check that 
$|\nabla g_0|(t, z)=|\nabla g|(t)$ for any $t \in (0, \pi)$, any $z \in Z$, applying (\ref{eq:coarea2}) for $f=|\nabla g|^p$ yields
$$
\lambda_{1, p}(Y, \dist, \meas)\le \int_Y|\nabla g_0|^p\di\meas = \int_0^{\pi}|\nabla g|^p\di\upsilon_N.
$$
Taking the infimum for $g$ with Proposition \ref{prop:char model sp} yields 
$$(\lambda_{1, p}(Y, \dist, \meas))^{1/p}=(\lambda_{1, p}([0, \pi], \dist, \upsilon_N))^{1/p}=(\lambda_{1,p}^N)^{1/p},$$ 
because $c_p(g_0)=\|g_0\|_{L^p}=1$.

We are now in a position to finish the proof of Corollary~\ref{cor;shperical suspension}.
The proof is done by contradiction via a standard compactness argument. Assume that the assertion is false.
Then there exist $\epsilon>0$, $p_i \in [1, \infty]$, $q_i \in [1, \infty]$ and $RCD^*(N-1, N)$-spaces $(X_i,\dist_i,\meas_i)$ with $\supp\meas_i=X_i$ and $\meas_i(X_i)=1$ such that 
$$
\lim_{i \to \infty}\left|\left( \lambda_{1, p_i}(X_i,\dist_i,\meas_i)\right)^{1/p_i}-\bigl(\lambda_{1, p_i}^N\bigr)^{1/p_i}\right|=0
$$ 
and 
$$
\left|\left( \lambda_{1, q_i}(X_i,\dist_i,\meas_i)\right)^{1/q_i}-\bigl(\lambda_{1, q_i}^N\bigr)^{1/q_i}\right|\geq\epsilon.
$$ 
By the sequential compactness of $\mathcal{M}(N)$, without loss of generality we can assume (after embedding isometrically $(X_i,\dist_i)$ into a common metric
space $(X,\dist)$),
that $X_i=X$, $\dist_i=\dist$ and that the measured Gromov-Hausdorff limit $(X,\dist,\meas)$ of the spaces $(X,\dist,\meas_i)$
exists, and is an $RCD^*(N-1, N)$-space.
We assume also that the limits $p,\, q \in [1, \infty]$ of $p_i,\, q_i$ exist.
Then Theorem~\ref{thm:p-spec} and Theorem~\ref{thm:gros} yield that
$$
\left(\lambda_{1, p}(X,\dist,\meas)\right)^{1/p}=\bigl(\lambda_{1, p}^N\bigr)^{1/p}
$$
and that
$$
\left(\lambda_{1, q}(X,\dist,\meas)\right)^{1/q}\neq\bigl(\lambda_{1, q}^N\bigr)^{1/q}.
$$
This contradicts Theorem~\ref{thm:almost rigidity}  with the argument above.
\end{proof}

\end{document}